\documentclass[11pt]{amsart}

\usepackage{amssymb}
\usepackage[all,cmtip]{xy}
\usepackage{mathrsfs}

\newtheorem{theorem}{Theorem}[section]
\newtheorem{lem}[theorem]{Lemma}
\newtheorem{cor}[theorem]{Corollary}
\newtheorem{prop}[theorem]{Proposition}

\theoremstyle{definition}
\newtheorem{defi}[theorem]{Definition}
\newtheorem{example}[theorem]{Example}

\theoremstyle{remark}
\newtheorem{rem}[theorem]{Remark}

\newcommand{\dstirling}[2]{\genfrac{[}{]}{0pt}{0}{#1}{#2}}   

\newcommand{\BC}{\mathbb{C}}            
\newcommand{\BZ}{\mathbb{Z}}             
\newcommand{\BN}{\mathbb{N}}            

\newcommand{\Glie}{\mathfrak{g}}             
\newcommand{\Hlie}{\mathfrak{h}}          
\newcommand{\qaf}{U_q(\hat{\mathfrak{g}})}    
\newcommand{\Gaff}{\hat{\mathfrak{g}}}   

\newcommand{\BY}{Y^{\mathrm{ab}}}   
\newcommand{\CR}{\mathcal{R}}    
\newcommand{\BR}{\mathbf{R}}      
\newcommand{\ttimes}{\,\overline{\otimes}\,}    

\newcommand{\BQ}{\mathbf{Q}}                
\newcommand{\BP}{\mathbf{P}}            
\newcommand{\wt}{\mathrm{wt}}         
\newcommand{\Bp}{\mathbf{p}}    

\usepackage{color}

\allowdisplaybreaks                

\title[unitriangular R-matrices and Theta series]{unitriangular R-matrices of quantum affine algebras and Yangians via Theta series}
\author{Huafeng Zhang}
\address{HZ: CNRS, UMR 8524-Laboratoire Paul Painlev\'e, Univ. Lille, F-59000 Lille, France}
\email{huafeng.zhang@univ-lille.fr}

\begin{document}
\begin{abstract} The universal R-matrix of the quantum affine algebra associated to a finite-dimensional simple complex Lie algebra admits a Gauss decomposition into an uper unitriangular part, an abelian part, and a lower unitriangular part.  In this paper, we provide a simple conjugation formula for the unitriangular R-matrices with one tensor factor evaluated at an arbitrary finite-dimensional representation of the quantum affine algebra. Our formula involves the T-series of Frenkel--Hernandez and the Theta series introduced in a previous work.

We also extend our conjugation formula to the Yangian case, making use of associators for triple tensor product representations of shifted Yangians.  \end{abstract}

\maketitle


\section{Introduction}

Fix $\Glie$ to be a finite-dimensional simple complex Lie algebra and $q \in \BC^{\times}$ that is not a root of unity. Consider the untwisted quantum affine algebra $\qaf$ of zero central charge and without derivation operator. As a Drinfeld--Jimbo quantum group, it is a quasi-triangular Hopf algebra and admits therefore a universal R-matrix in a suitable completion of the tensor square of $\qaf$. 

The Hopf algebra $\qaf$ possesses a $\BC^{\times}$-action by dilatation, so that its universal R-matrix can be upgraded to a power series $R(z)$ in a formal variable $z$ with coefficients in the completed tensor square. We avoid the completion by evaluating one tensor factor of the power series at a finite-dimensional representation $\varphi: \qaf \longrightarrow \mathrm{End}(W)$. The resulting series in the ordinary tensor products are called {\it L-operators}:
$$ (\varphi \otimes \mathrm{Id})(R(z)) \in \mathrm{End}(W) \otimes \qaf[[z]], \qquad (\mathrm{Id}\otimes \varphi)(R(z)) \in \qaf[[z]] \otimes \mathrm{End}(W) $$
By the defining properties of the universal R-matrix, each L-operator satisfies a modified version of the quantum Yang--Baxter equation. Among its main applications let us mention: the formulation and proof of Knizhnik--Zamolodchikov difference equations for vertex operators between highest weight representations of the quantum affine algebra \cite{FR0}; the construction of quantum integrable models generalizing the six-vertex model and Baxter polynomiality for their spectra \cite{FR1,FH}; the isomorphism from R-matrix realization of the quantum affine algebra to its Drinfeld new realization \cite{FM,JLM1,JLM2} that is inverse to the Ding--Frenkel isomorphism \cite{DF}.  

Khoroshkin--Tolstoy \cite{TK} and Damiani \cite{Damiani} computed the universal R-matrix as infinite ordered products. From their formulas we observe a triple product decomposition
$$ R(z) = R^+(z) \times R^0(z) \times R^-(z). $$
Recall that the quantum affine algebra is weight graded by the root lattice $\BQ$ of the underlying simple Lie algebra $\Glie$. 
At the right-hand side, the third factor $R^-(z)$ is {\it lower unitriangular} in the sense that it is an infinite sum $\sum_{\beta \in \BQ_+} R^-_{\beta}(z)$ over the positive root cone $\BQ_+$, where each component $R_{\beta}^-(z)$ is a power series in $z$ with coefficients in the tensor product $\qaf_{-\beta} \otimes \qaf_{\beta}$ and $R_0^-(z) = 1 \otimes 1$. Likewise, the first factor $R^+(z)$ is upper unitriangular. The middle factor $R^0(z)$ has coefficients in the completed tensor square of a commutative subalgebra of $\qaf$, the Drinfeld--Cartan subalgebra. 

 Both unitriangular factors $R^{\pm}(z)$ are infinite ordered products of $q$-exponentials. Understanding their evaluations $(\varphi\otimes \psi)(R^{\pm}(z))$ at two representations $\varphi$ and $\psi$ of the quantum affine algebra is a highly non-trivial task. In the case $\Glie = \mathfrak{sl}_2$, explicit computations for Verma modules were performed by Khoroshkin--Stolin--Tolstoy \cite{KST}, and integral formulas for $R^{\pm}(z)$ were obtained by Ding--Pakuliak--Khoroshkin \cite{DPK}. For simply-laced $\Glie$ and suitably chosen finite-dimensional irreducible representations, up to Drinfeld--Cartan factors, Okounkov--Smirnov \cite{OS} identified $(\varphi\otimes \psi)(R^{\pm}(z))$ with stable envelopes in the equivariant K-theory of Nakajima quiver varieties. In general types, Hernandez \cite{H1} proposed the notion of algebraic stable map and showed that $(\varphi \otimes \psi)(R^{\pm}(z))$ for a large class of representations (including all the finite-dimensional irreducible ones) are typical examples of algebraic stable maps.    

 \medskip

 The main result of this paper is a simple conjugation formula for unitriangular L-operators associated to an arbitrary finite-dimensional representation $\varphi$ of $\qaf$.
 We shall explain the formula for $(\mathrm{Id} \otimes \varphi)(R^-(z))$, as the other cases are parallel. 
 
 Let $r$ be the rank of the simple Lie algebra $\Glie$. In the Drinfeld new realization, the quantum affine algebra has four families of generating series:
 $$ x_i^{\pm}(z) \in \qaf[[z,z^{-1}]] \quad \textrm{and} \quad \phi_i^{\pm}(z) \in \qaf[[z^{\pm 1}]]\quad \textrm{for $1\leq i \leq r$}. $$
 The coefficients of the power series $\phi_i^{\pm}(z)$ generate the aforementioned commutative Drinfeld--Cartan subalgebra. To a tuple $\Bp = (\Bp_i(z))_{1\leq i \leq r}$ of complex polynomials of constant term 1, Frenkel--Hernandez \cite{FH} attached a power series $T_{\Bp}(z)$ in $z$ of constant term 1 and with coefficients in the Drinfeld--Cartan subalgebra. Compared to $\phi$-series, T-series satisfy easier commutation relations with the generating series:
 $$ T_{\Bp}(z) x_i^-(w) = \Bp_i(\frac{z}{w})x_i^-(w) T_{\Bp}(z), \quad x_i^+(w) T_{\Bp}(z) = \Bp_i(\frac{z}{w}) T_{\Bp}(z) x_i^+(w).  $$
 Define the Theta series $\Theta_{\Bp}(z)$ by factorizing the coproduct \cite{Z2}
 $$ \Delta(T_{\Bp}(z)) = (1 \otimes T_{\Bp}(z)) \times \Theta_{\Bp}(z) \times (T_{\Bp}(z) \otimes 1) \in \qaf^{\otimes 2}[[z]]. $$
 As in the case of the lower unitriangular R-matrix $R^-(z)$, the Theta series is an infinite sum $\sum_{\beta \in \BQ_+} \Theta_{\Bp,\beta}(z)$ of power series in $z$ with coefficients in $\qaf_{-\beta}\otimes \qaf_{\beta}$ and its zero component is $1\otimes 1$. However, a sharp difference is that each component $\Theta_{\Bp,\beta}(z)$ is a {\it polynomial} in $z$, while for $R_{\beta}^-(z)$ this is rarely true.   

The only reason why we need the representation $\varphi$ to be finite-dimensional is the existence of a tuple $\Bp$ of polynomials such that, for each $1\leq i \leq r$, the formal Laurent series $\varphi(x_i^+(z))$ is annihilated by the Laurent polynomial $\Bp_i(z^{-1})$. Our conjugation formula then reads as follows (Theorem \ref{thm: conjugation formula qaf}) 
 $$ (\mathrm{Id}\otimes \varphi)(R^-(z)) = (T_{\Bp}(z) \otimes 1)^{-1} \times (\tau_z \otimes \varphi)(\Theta_{\Bp}(1)) \times (T_{\Bp}(z) \otimes 1).  $$
 In the middle term at the right-hand side, $\tau_z$ is the algebra homomorphism from $\qaf$ to $\qaf[z,z^{-1}]$ arising from the $\BC^{\times}$-action on the quantum affine algebra, and $\Theta_{\Bp}(1)$ is the Theta series $\Theta_{\Bp}(z)$ evaluated at $z = 1$, which is well-defined by polynomiality. 

Our conjugation formula follows from the polynomiality of Theta series \cite{Z2} and an intertwining property of the unitriangular R-matrices due to Khoroshkin--Tolstoy \cite{KT} and Enriquez--Khoroshkin--Pakuliak \cite{EKP}: the Drinfeld--Jimbo coproduct conjugated by $R^-(z)$ is a Drinfeld formal coproduct under which T-series are grouplike. 

As a direct application of the conjugation formula, we establish the rationality of $(\varphi\otimes \psi)(R^{\pm}(z))$ for arbitrary finite-dimensional representations $\varphi$ and $\psi$ and estimate their poles in terms of annihilators of the formal Laurent series $\varphi(x_i^-(z))$ and $\psi(x_i^+(z))$.

Consider the Yangian $Y(\Glie)$ associated to the simple Lie algebra $\Glie$. Its universal R-matrix has a triple product decomposition into two unitriangular factors and an abelian factor, as shown by Gautam--Toledano Laredo--Wendlandt \cite{GTLW}. Our conjugation formula has a natural Yangian counterpart. Its proof utilizes representation theory of shifted Yangians, notably the non-strict associators for triple tensor products.  

\medskip

One may think of Theta series as a simplified version of unitriangular R-matrices. In many situations, the former plays the role of the latter as indicated by the conjugation formulas and by our previous work \cite{Z2} on R-matrices of asymptotic representations. It is desirable to have simpler factorization formulas for Theta series than those of Khoroshkin--Tolstoy and Damiani for the universal R-matrix of the quantum affine algebra. As shown in \cite{Milot}, Theta series are {\it finite} ordered products of $q$-exponentials for $U_q(\widehat{\mathfrak{sl}}_2)$ and $ U_q(\widehat{\mathfrak{sl}}_3)$, and of ordinary exponentials for $Y(\mathfrak{sl}_{r+1})$. We comment that no such factorization formula is known for the Yangian universal R-matrix \cite{GTLW,AGW}.

In the proof of Yangian conjugation formula, we extend unitriangular R-matrices to shifted Yangians. Recently, a family of representations of shifted Yangians were obtained on critical cohomologies of framed quiver varieties with potentials \cite{COZZ}. We expect the evaluations of unitriangular R-matrices at these representations to be critical stable envelopes therein after a Drinfeld--Cartan normalization. In another direction, the abelian factor in the full R-matrices of representations of shifted Yangians is missing. Similar questions can be posed for shifted quantum affine algebras \cite{FT}.

\medskip

This paper is organized as follows. Section \ref{sec: qaf} collects basic properties of the quantum affine algebra and its universal R-matrix. In Section \ref{sec: conjugation formula qaf} we prove the conjugation formula for unitriangular L-operators of the quantum affine algebra. In Section \ref{sec: shifted Yangians} we recall known facts on representation theory of shifted Yangians:  Theta series and R-matrices for one-dimensional representations. Section \ref{sec: spectral parameter coproduct} studies the compatibility between shifted coproducts and spectral parameter automorphisms. We extend unitriangular R-matrices to shifted Yangians in Section \ref{sec: uni R shifted Yangian} and then prove the conjugation formula for unitriangular L-operators of shifted Yangians in Section \ref{se: conjugation formula shifted Yangians}. 

\medskip

\noindent {\bf Acknowledgments:} The author wishes to thank Andrea Appel and Jae-Hoon Kwon for useful discussions. He acknowledges the support of the CDP C2EMPI, as well as the French State under the France-2030 programme, the University of Lille, the Initiative of Excellence of the University of Lille, the European Metropolis of Lille for their funding and support of the R-CDP-24-004-C2EMPI project.

\section{Generalities on quantum affine algebras} \label{sec: qaf}
In this section, we review the basic properties of quantum affine algebras: two equivalent realizations, universal R-matrix, modified Drinfeld--Cartan series and Theta series. The ground field is $\BC$, and $\BN := \BZ_{\geq 0}$.

Fix $\Glie$ to be a finite-dimensional simple Lie algebra. Let $\Hlie$ be a Cartan subalgebra of $\Glie$, and $I := \{1,2,\cdots,r\}$ be the set of Dynkin nodes. The dual Cartan subalgebra $\Hlie^*$ admits a basis of {\it simple roots} $(\alpha_i)_{i \in I}$ and a non-degenerate invariant symmetric bilinear form $(,): \Hlie^* \times \Hlie^* \longrightarrow \BC$ normalized in such a way that the $d_i := \frac{(\alpha_i,\alpha_i)}{2}$ for $i \in I$ are co-prime positive integers in $\{1,2,3\}$. We have the Cartan matrix $(c_{ij})_{i,j\in I}$ and the symmetrized Cartan matrix $(b_{ij})_{i,j\in I}$ with integer entries defined by
$$c_{ij} := \frac{2(\alpha_i,\alpha_j)}{(\alpha_i,\alpha_i)},\quad b_{ij} := (\alpha_i,\alpha_j) \quad \textrm{for $i, j \in I$}. $$
We shall need the root lattice in $\Hlie^*$ and its cones 
$$ \BQ := \bigoplus_{i\in I} \BZ \alpha_i,\quad \BQ_+ := \sum_{i\in I} \BN \alpha_i,\quad \BQ_- := -\BQ_+.  $$

Fix $q \in \BC^{\times}$, which is not a root of unity. For $t \in \BZ$ and $n \in \BN$, we set
$$  [t]_q := \frac{q^t-q^{-t}}{q-q^{-1}}, \qquad \dstirling{t}{n}_q := \prod_{m=1}^n [t-m+1]_q,\qquad (t)_q := \frac{q^{2t}-1}{q^2-1}.  $$
Set $q_i := q^{d_i}$ for $i \in I$.
It is known that the $I\times I$-matrix $([b_{ij}]_q)_{i,j\in I}$ is invertible. Let $(\widetilde{B}_{ij}(q))_{i,j\in I}$ denote its inverse. 

\subsection{The quantum affine algebra}
Let $\theta = \sum_{i=1}^r a_i \alpha_i \in \BQ_+$ be the highest positive root of $\mathfrak{g}$.
We enlarge the Cartan matrix $(c_{ij})_{1\leq i,j \leq r}$ to a non-twisted affine Cartan matrix $(c_{ij})_{0\leq i,j\leq r}$ as follows:
$$c_{00} := 2, \quad c_{i0} = -\frac{2(\alpha_i, \theta)}{(\alpha_i, \alpha_i)},\quad  c_{0i} := -\frac{2(\theta, \alpha_i)}{(\theta, \theta)}\quad \mathrm{for}\ 1 \leq i \leq r. $$
 Define $q_0 := q^{d_0}$ where $d_0 := \frac{(\theta,\theta)}{2}$.
 
In the Drinfeld--Jimbo realization, the {\it quantum affine algebra} of zero central charge and without derivation operator, denoted by $\qaf$, is the associative algebra generated by $E_i, F_i, K_i^{\pm 1}$ for $0 \leq i \leq r$ and subject to the following relations for $0 \leq i, j \leq r$:
\begin{gather*}
 K_iK_i^{-1} = 1 = K_i^{-1}K_i,\quad K_i K_j = K_j K_i, \quad K_0 K_1^{a_1} K_2^{a_2} \cdots K_r^{a_r} = 1, \\
K_i E_j = q_i^{c_{ij}} E_j K_i,\quad K_i F_j = q_i^{-c_{ij}} F_j K_i, \quad [E_i, F_j] = \delta_{ij} \frac{K_i - K_i^{-1}}{q_i-q_i^{-1}}, \\
\sum_{s=0}^{1-c_{ij}} (-1)^s \dstirling{1-c_{ij}}{s}_{q_i} E_i^{1-c_{ij}-s} E_j E_i^s = 0 = \sum_{s=0}^{1-c_{ij}} (-1)^s \dstirling{1-c_{ij}}{s}_{q_i} F_i^{1-c_{ij}-s} F_j F_i^s  \quad \textrm{if $i \neq j$}.
\end{gather*}
It has a Hopf algebra structure with the coproduct given by:
\begin{equation*} 
\Delta(E_i) = E_i \otimes 1 + K_i \otimes E_i,\quad \Delta(F_i) = 1 \otimes F_i + F_i \otimes K_i^{-1},\quad \Delta(K_i) = K_i \otimes K_i.
\end{equation*}

With respect to the conjugate action of the Cartan generators $K_i$ for $1\leq i\leq r$, the Hopf algebra $\qaf$ admits a $\BQ$-grading, called {\it weight grading}: an element $x \in \qaf$ is of weight $\beta \in \BQ$, and we write $\wt(x) = \beta$, if $K_i x K_i^{-1} = q^{(\alpha_i,\beta)} x$ for $1\leq i \leq r$. Let $\qaf_{\beta}$ denote the subspace of elements in $\qaf$ of weight $\beta$.

The Hopf algebra $\qaf$ is also $\BZ$-graded by declaring the generator $E_0$ to be of degree 1, the generator $F_0$ to be of degree $-1$, and all the other Drinfeld--Jimbo generators to be of degree 0. 
Let $z$ be a formal variable, called {\it spectral parameter}. The $\BZ$-grading induces an algebra homomorphism
\begin{equation}   \label{def: spectral para qaf}
    \tau_z: \qaf \longrightarrow \qaf[z,z^{-1}], \quad x \mapsto z^{\deg (x)} x \quad \textrm{for $x$ homogeneous}.
\end{equation}
One specializes $z$ at nonzero complex numbers to get a one-parameter family $(\tau_a)_{a \in \BC^{\times}}$ of Hopf algebra automorphisms of $\qaf$. By replacing $z^{\deg(x)}$ with $z^{-\deg(x)}$, we get another algebra homomorphism from $\qaf$ to $\qaf[z,z^{-1}]$, denoted by $\tau_{z^{-1}}$.

The quantum affine algebra $\qaf$ has a second presentation, called Drinfeld new realization \cite{Dri88}. Its generators are $x_{i,m}^{\pm}$ and $\phi_{i,m}^{\pm}$ for $(i,m) \in I \times \BZ$. Its defining relations are as follows for $(i,j,m,n) \in I^2 \times \BZ^2$ and $\varepsilon \in \{+, -\}$:
\begin{gather*} 
\phi_{i,m}^+ = 0 \ \mathrm{if}\ m < 0,\quad \phi_{i,m}^- = 0 \ \mathrm{if}\ m > 0, \quad \phi_{i,0}^+ \phi_{i,0}^- = 1,  \\
[\phi_{i,m}^{\pm}, \phi_{j,n}^{\varepsilon}] = 0, \quad [x_{i,m}^+, x_{j,n}^-] = \delta_{ij} \frac{\phi_{i,m+n}^+-\phi_{i,m+n}^-}{q_i-q_i^{-1}},   \\
\phi_{i,m+1}^{\varepsilon} x_{j,n}^{\pm} -q^{\pm b_{ij}} \phi_{i,m}^{\varepsilon} x_{j,n+1}^{\pm} = q^{\pm b_{ij}} x_{j,n}^{\pm} \phi_{i,m+1}^{\varepsilon} - x_{j,n+1}^{\pm} \phi_{i,m}^{\varepsilon}, \\
x_{i,m+1}^{\pm} x_{j,n}^{\pm} -q^{\pm b_{ij}} x_{i,m}^{\pm} x_{j,n+1}^{\pm} = q^{\pm b_{ij}} x_{j,n}^{\pm} x_{i,m+1}^{\pm} - x_{j,n+1}^{\pm} x_{i,m}^{\pm},  \\
\sum_{s=0}^{1-c_{ij}} (-1)^s \dstirling{1-c_{ij}}{s}_{q_i}(x_{i,0}^{\pm})^{1-c_{ij}-s} x_{j,0}^{\pm} (x_{i,0}^{\pm})^s = 0 \quad \mathrm{if}\ i \neq j. 
\end{gather*}
Beck \cite{Beck} constructed a surjective algebra homomorphism from the Drinfeld new realization of $\qaf$ to its Drinfeld--Jimbo realization. Later Damiani \cite{Damiani2} completed the injectivity. Under the correspondence of Beck, we have $$E_i = x_{i,0}^+, \quad F_i = x_{i,0}^-,\quad K_i = \phi_{i,0}^+ \quad \textrm{for $1\leq i \leq r$}.$$
The weight grading and the $\BZ$-grading in terms of Drinfeld generators are given by
\begin{gather*}
 \wt(x_{i,m}^{\pm}) = \pm\alpha_i, \quad  \wt(\phi_{i,m}^{\pm}) = 0, \quad \deg(x_{i,m}^{\pm}) = \deg(\phi_{i,m}^{\pm}) = m.
\end{gather*}

Define five subalgebras of $\qaf$ by generating subsets:
\begin{gather*}
    U_q^+(\Gaff) := \langle x_{i,m}^+ \rangle_{(i,m)\in I \times \BZ}, \quad U_q^0(\Gaff) := \langle \phi_{i,m}^{\pm} \rangle_{(i,m)\in I \times \BZ}, \quad U_q^-(\Gaff) := \langle x_{i,m}^-\rangle_{(i,m) \in I \times \BZ}, \\
    U_q^{\geq}(\Gaff) := \langle x_{i,m}^+, \phi_{i,m}^{\pm} \rangle_{(i,m)\in I\times \BZ}, \quad U_q^{\leq}(\Gaff) := \langle x_{i,m}^-, \phi_{i,m}^{\pm} \rangle_{(i,m)\in I\times \BZ}.
\end{gather*}
The weight grading and the $\BZ$-grading restrict to these subalgebras. The subalgebra $U_q^0(\Gaff)$ is commutative, and commonly called the {\it Drinfeld--Cartan subalgebra}. 

Let $u$ be a formal variable. Define the generating series of $\qaf$ by:
\begin{equation}  \label{def: gene series qaf}
    x_i^{\pm}(u) := \sum_{m\in \BZ} x_{i,m}^{\pm} u^m, \quad \phi_i^{\pm}(u) := \sum_{m\in \BZ} \phi_{i,m}^{\pm} u^m \quad \textrm{for $i \in I$}.
\end{equation}
Then $\phi_i^{\pm}(u)$ are power series in $u^{\pm 1}$ with invertible constant terms $\phi_{i,0}^{\pm}$ and with coefficients in the commutative Drinfeld--Cartan subalgebra. Define the elements $h_{i,s}$ of the Drinfeld--Cartan subalgebra, for $i \in I$ and $s \in \BZ_{\neq 0}$, by the equations
\begin{equation*}  
\phi_i^{\pm}(u) =  \phi_{i,0}^{\pm} \exp\left(\pm (q-q^{-1}) \sum_{\pm s>0} h_{i,s} u^s\right) \in U_q^0(\Gaff)[[z^{\pm 1}]].
\end{equation*}
We follow the convention of \cite[(2.2)]{FR1} so that our $h_{i,s}$ is $[d_i] h_{i,s}$ in \cite{Beck}. 
\subsection{Gauss decomposition of the universal R-matrix} The quantum affine algebra $\qaf$ is a quasi-triangular Hopf algebra. In this subsection, we collect basic properties of its universal R-matrix $R_{\Gaff}$ from \cite{Damiani, KT, FM, EKP}. 

By \cite[Theorem 2]{Damiani}, the universal R-matrix $R_{\Gaff}$ is a quadruple product
$$ R_{\Gaff} = R^+ \times R^0 \times  R^- \times q^{-t_{\infty}} $$
in a completion of the tensor square of $\qaf$, which we do need here.
At the right hand side, the first three factors $R^+,\, R^0$ and $R^-$ arise from the linearly ordered set of positive roots of the affine Kac--Moody algebra $\Gaff$ partitioned into three subsets in \cite[Definition 4]{Damiani}, and $t_{\infty}$ in the fourth factor is the canonical element associated with the bilinear form on $\Hlie$. Our primary concern in this paper is $R^{\pm}$. Their precise formulas will be recalled in the proof of Proposition \ref{prop: unitriangular R qaf} below and will not be used elsewhere.

Following \cite[(2.25)]{KST}, we move the fourth $q^{-t_{\infty}}$ to the second factor $R^0$ and obtain a triple product, which is called a {\it Gauss decomposition} as in \cite{DF}:
$$ R_{\Gaff} = R^+ \times (R^0 q^{-t_{\infty}}) \times (q^{t_{\infty}} R^- q^{-t_{\infty}}).  $$
Add a spectral parameter $z$ by applying $\tau_z$ of Eq.\eqref{def: spectral para qaf} to the first tensor factors:
$$ R^+(z) := (\tau_z \otimes \mathrm{Id})(R^+), \quad R^0(z) := (\tau_z \otimes \mathrm{Id})(R^0),\quad \BR^-(z) := (\tau_z \otimes \mathrm{Id})(q^{t_{\infty}} R^- q^{-t_{\infty}}).  $$
Let $\BR^+(z) := R_{21}^+(z)$ be obtained from $R^+(z)$ by permuting its two tensor factors; the flip will simplify the formulas in Proposition \ref{prop: unitriangular R qaf} below. Then we have:
\begin{equation}  \label{equ: Gauss qaf}
    (\tau_z \otimes \mathrm{Id})(R_{\Gaff}) = \BR_{21}^+(z) \times R^0(z) q^{-t_{\infty}} \times \BR^-(z).
\end{equation}

Let us make sense of the space where the factors $\BR^{\pm}(z)$ live.
\begin{defi}\label{defi: completion qaf}
    Given two $\BQ$-graded algebras $\mathbb{A} = \oplus_{\alpha\in \BQ} \mathbb{A}_{\alpha}$ and $\mathbb{B} = \oplus_{\alpha\in \BQ} \mathbb{B}_{\alpha}$, we define their {\it 0-completed tensor product} $\mathbb{A}\ttimes \mathbb{B}$ to be the following vector space
    $$ \mathbb{A}\ttimes \mathbb{B} :=  \prod_{\beta \in \BQ_+} (\mathbb{A}_{-\beta} \otimes \mathbb{B}_{\beta}) \subset \prod_{\alpha,\beta \in \BQ} \mathbb{A}_{\alpha} \otimes \mathbb{B}_{\beta}.$$
    It is naturally an algebra that contains the subalgebra $\sum_{\beta\in \BQ_+} \mathbb{A}_{-\beta} \otimes \mathbb{B}_{\beta}$ of the ordinary tensor product algebra $\mathbb{A} \otimes \mathbb{B}$.
\end{defi}
We use 0-completion to distinguish with the  completion $\mathbb{A}\, \widetilde{\otimes}\,\mathbb{B}$ of the {\it full} tensor product algebra $\mathbb{A} \otimes \mathbb{B}$ defined in \cite[\S 8.3]{Z2}. The full completion contains the 0-completion. In this paper, working with the 0-completion is sufficient.

The next result, due to \cite{Damiani,Damiani2,KT,EKP}, summarizes the key properties of $\BR^{\pm}(z)$ to be used later on. We provide relevant details with references in the proof.  
\begin{prop}\label{prop: unitriangular R qaf}
\begin{itemize}
    \item[(1)] Both $\BR^{\pm}(z)$ are power series in $z$ with coefficients in the 0-completed tensor product $U_q^-(\Gaff)\ttimes U_q^+(\Gaff)$, and are {\it unitriangular}:
\begin{gather*}
    \BR^{\pm}(z) = \sum_{\beta \in \BQ_+} \BR_{\beta}^{\pm}(z) \quad \textrm{such that} \\
    \BR_{\beta}^{\pm}(z) \in (U_q^-(\Gaff)_{-\beta} \otimes U_q^+(\Gaff)_{\beta})[[z]] \quad \textrm{for}\quad \beta \in \BQ_+ \quad \textrm{and}\quad \BR_0^{\pm}(z)  = 1 \otimes 1.
\end{gather*}
\item[(2)] For $i \in I$ and $s\in \BZ_{\neq 0}$, we have the four intertwining equations of Laurent series in $z$ with coefficients in the 0-completed tensor product $U_q^{\leq}(\Gaff)\ttimes U_q^{\geq}(\Gaff)$: 
\begin{align*}
    (h_{i,s}\otimes z^s + 1 \otimes h_{i,s}) \times \BR^-(z) &= \BR^-(z) \times (\tau_z \otimes \mathrm{Id})\circ \Delta(h_{i,s}), \\
    \BR^+(z) \times (h_{i,s} \otimes 1 + z^s \otimes h_{i,s}) &= (\mathrm{Id} \otimes \tau_z)\circ \Delta(h_{i,s}) \times \BR^+(z), \\
    (h_{i,s} \otimes 1 + z^{-s} \otimes h_{i,s}) \times \BR^-(z) &= \BR^-(z) \times (\mathrm{Id}\otimes \tau_{z^{-1}})\circ \Delta(h_{i,s}), \\
    \BR^+(z) \times (h_{i,s} \otimes z^{-s} + 1\otimes h_{i,s}) &= (\tau_{z^{-1}} \otimes \mathrm{Id})\circ \Delta(h_{i,s}) \times \BR^+(z).
\end{align*}
\end{itemize}    
\end{prop}
\begin{proof}
We use the notations of \cite{Damiani}.
    Let $\Phi_{0,+} \subset \BQ_+$ denote the set of positive roots of the simple Lie algebra $\Glie$. By \cite[Definitions 4 \& 7]{Damiani}, there exist two linear orders on the sets $\BZ_{>0} \times \Phi_{0,+}$ and $\BZ_{\geq 0} \times \Phi_{0,+}$ respectively, with respect to which we have the following infinite ordered product formulas for $\BR^{\pm}(z)$:
    \begin{align*}
        \BR^-(z) &= \prod_{(n,\beta)\in \BZ_{>0}\times \Phi_{0,+}}^{\succ} \exp_{q_{\beta}}\left( (q_{\beta}^{-1}-q_{\beta}) K_{\beta} E_{n\delta-\beta} \otimes F_{n\delta-\beta} K_{\beta}^{-1} z^n \right), \\
        \BR^+(z) &= \prod_{(n,\beta)\in \BZ_{\geq 0}\times \Phi_{0,+}}^{\succ} \exp_{q_{\beta}}\left( (q_{\beta}^{-1}-q_{\beta}) F_{n\delta+\beta} \otimes E_{n\delta+\beta} z^n \right).
    \end{align*}
    Here, for $\beta \in \Phi_{0,+}$, we have $q_{\beta} := q^{\frac{(\beta,\beta)}{2}} \in \{q,q^2,q^3\}$. The $q$-exponential $\exp_q(x)$ is  
    $$ \exp_q(x) := 1 + \sum_{m=1}^{+\infty}  \frac{1}{(1)_q(2)_q \cdots (m)_q} x^m \quad \textrm{with }\quad (m)_q := \frac{q^{2m}-1}{q^2-1}.  $$
    For $\BR^-(z)$ we used the following conjugation relation 
    $$ q^{t_{\infty}} (x \otimes y) q^{-t_{\infty}} = K_{\gamma} x \otimes y K_{\beta} \quad \textrm{for $x \in \qaf_{\beta}$ and $y \in \qaf_{\gamma}$.} $$
        Each root vector $E_{n\delta+\beta}$, for $n \in \BZ$ and $\beta \in \pm \Phi_{0,+}$, is of weight $\beta$ and of degree $n$. Similarly, $F_{n\delta+\beta}$ is of weight $-\beta$ and of degree $-n$. By \cite[Proposition 9.3]{Damiani2}, we have $K_{\beta}E_{n\delta-\beta},\, F_{n\delta+\beta} \in U_q^-(\Gaff)$ and $F_{n\delta-\beta} K_{\beta}^{-1},\, E_{n\delta+\beta} \in U_q^+(\Gaff)$. This proves Part (1).
    
    Part (2) is a consequence of \cite[Theorem 8.1]{KT} and \cite[Theorem 3.8]{EKP}: both $R^+_{21}$ and $(R^-q^{-t_{\infty}})^{-1}$ intertwine $h_{i,s}\otimes 1 + 1 \otimes h_{i,s}$ with $\Delta(h_{i,s})$.
\end{proof}

We call $\BR^{\pm}(z)$ {\it unitriangular R-matrices} of the quantum affine algebra.
In Part (2) of the proposition, the first and third equations are equivalent because the coproduct $\Delta$ preserves the $\BZ$-grading of $\qaf$ and the coefficients of $\BR^-(z)$ are of total degree 0. Similarly, the remaining two equations are equivalent. We list all these equations because they will lead to four different conjugation formulas of $\BR^{\pm}(z)$.

\subsection{T-series and Theta series} In this subsection, we recall from \cite{FH} the T-series of Frenkel--Hernandez and from \cite{Z2} the construction of Theta series. These auxiliary series will play a key role in our conjugation formulas.

To each Dynkin node $i \in I$ are attached the positive T-series $T_i^+(z)$ and the negative T-series $T_i^-(z)$ as follows:
\begin{equation} 
   T_i^{\pm}(z) := \exp\left(\sum_{s > 0}  \sum_{j\in I} \frac{\widetilde{B}_{ji}(q^s)}{[s]_q} h_{j,\pm s} z^s\right) \in U_q^0(\Gaff)[[z]].  \label{def: T}
\end{equation}
Our $T_i^{\pm}(z)$ are precisely the series $\mathbf{T}_i^{\pm}(z^{\pm 1})$ in \cite[(3.13)]{HZ2}, and $T_i^-(z)$ is the original Frenkel--Hernandez series $T_i(z)$ in \cite[Proposition 5.5]{FH}. We use power series in the spectral parameter $z$ instead of formal Laurent series in the variable $u$ as in Eq.\eqref{def: gene series qaf} because later we will compute the R-matrices $\BR^{\pm}(z)$ in terms of T-series.

Compared with the Drinfeld--Cartan series, T-series have simpler commutation relations with the Drinfeld generators. By \cite[(3.16)]{HZ2} and its proof, we have the following relations in $\qaf[[z,u,u^{-1}]]$ for $i, j \in I$:
\begin{equation}  \label{rel: Ti x qaf}
    \begin{split}
        T_i^{\pm}(z)  x_j^-(u) = (1 - \delta_{ij} u^{\mp 1} z)  x_j^-(u)  T_i^{\pm}(z), \\
        x_j^+(u) T_i^{\pm}(z) = (1 - \delta_{ij} u^{\mp 1} z)T_i^{\pm}(z) x_j^+(u) .
    \end{split}
\end{equation}

More generally, we attach T-series to tuples of polynomials.
    We call a complex polynomial $p(u) \in \BC[u]$ a {\it Drinfeld polynomial} if its constant term is 1, or equivalently, $p(u)$ is a finite product of $1-au$ for $a \in \BC^{\times}$. Given an $I$-tuple $\mathbf{p} = (\mathbf{p}_i(u))_{i\in I}$ of Drinfeld polynomials, we factorize each component
    $$ \mathbf{p}_i(u) = (1-a_{i1}u) (1-a_{i2}u) \cdots (1-a_{i,k_i}u) \quad \textrm{with $k_i \in \BN$ and $a_{ij} \in \BC^{\times}$ for $1\leq j \leq k_i$.} $$
   The T-series associated with $\Bp$ are defined as \cite[Definition 9.2]{Z2}
   \begin{equation}  \label{def: T gen qaf}
       T_{\mathbf{p}}^{\pm}(z) := \prod_{i\in I} T_i^{\pm}( a_{i1}^{\mp 1}z) T_i^{\pm}( a_{i2}^{\mp 1} z) \cdots T_i^{\pm}(a_{i,k_i}^{\mp1} z) \in U_q^0(\Gaff)[[z]].
   \end{equation}
   The commutation relations \eqref{rel: Ti x qaf} are generalized as follows:
   \begin{equation}  \label{rel: Tp x qaf}
    \begin{split}
        T_{\mathbf{p}}^-(z) x_j^-(u) = \mathbf{p}_i(zu) x_j^-(u) T_{\mathbf{p}}^-(z), \quad T_{\mathbf{p}}^+(z) x_j^-(u) = \mathbf{p}_i^{*}(u^{-1}z) x_j^-(u) T_{\mathbf{p}}^+(z), \\
        x_j^+(u) T_{\mathbf{p}}^-(z) = \mathbf{p}_i(zu)T_{\mathbf{p}}^-(z) x_j^+(u), \quad  x_j^+(u) T_{\mathbf{p}}^+(z) = \mathbf{p}_i^*(u^{-1}z)  T_{\mathbf{p}}^+(z) x_j^+(u) .
    \end{split}
\end{equation}
  Here, $\mathbf{p}_i^*(u)$ denotes the Drinfeld polynomial
  $$ \mathbf{p}_i^*(u) = (1-a_{i1}^{-1}u) (1-a_{i2}^{-1}u) \cdots (1-a_{i,k_i}^{-1}u). $$

 Both T-series $T_{\Bp}^{\pm}(z)$ are invertible power series in $z$ of constant term 1. Factorizing properly their coproduct in the space $\qaf^{\otimes 2}[[z]]$ leads to the {\it Theta series} $\Theta_{\Bp}^{\pm}(z)$; see \cite[Definition 9.2]{Z2}:
\begin{equation}  \label{def: Theta}
    \Delta(T_{\Bp}^{\pm}(z)) = (1\otimes T_{\Bp}^{\pm}(z)) \times\Theta_{\Bp}^{\pm}(z) \times (T_{\Bp}^{\pm}(z) \otimes 1). 
\end{equation}  
 
 The next result was first proved for $\Theta^-$ in \cite[Theorem 9.5]{Z2}; the $\Theta^+$-case follows from the $\Theta^-$-case by \cite[Theorem 3.3]{HZ2}. Recently, Negu\c{t} extended it to more general quantum loop algebras \cite[Theorem 4.4]{Neg} using shuffle algebras.
\begin{theorem}\cite{Z2,HZ2}  \label{thm: Theta qaf}
    For $\Bp$ an $I$-tuple of Drinfeld polynomials, both Theta series $\Theta_{\Bp}^{\pm}(z)$ are unitriangular sums of polynomials in $z$ with coefficients in $U_q^-(\Gaff) \otimes U_q^+(\Gaff)$:
    \begin{gather*}
            \Theta_{\Bp}^{\pm}(z) = \sum_{\beta \in \BQ_+} \Theta_{\Bp,\beta}^{\pm}(z) \quad \textrm{with} \\
            \Theta_{\Bp,\beta}^{\pm}(z) \in U_q^-(\Gaff)_{-\beta}\otimes U_q^+(\Gaff)_{\beta}[z] \quad \textrm{for $\beta \in \BQ_+$ and} \quad \Theta_{\Bp,0}^{\pm}(z) = 1\otimes 1.
        \end{gather*}
\end{theorem}
As a consequence, one can specialize $z$ at an arbitrary nonzero complex number $a$ to get unitriangular elements $\Theta_{\Bp}^{\pm}(a)$ of the 0-completed tensor product $U_q^-(\Gaff) \ttimes U_q^+(\Gaff)$.

\begin{example}  \label{ex: Theta qaf}
    For $j\in I,\, a\in \BC^{\times}$, let $\Psi_{j,a}$ denote the $I$-tuple $(1-\delta_{ij} au)_{i\in I}$; this is a prefundamental $\ell$-weight from \cite[Definition 3.7]{HJ}.  Combining \cite[Example 9.6]{Z2} with \cite[Example 3.4]{HZ2}, we have the following formulas:
    \begin{align*}
        \sum_{n\geq 0}\Theta_{\Psi_{j,a},n\alpha_i}^+(z) &= \exp_{q_i^{-1}}\left( \delta_{ij} (q_i^{-1}-q_i) a^{-1}  x_{i,1}^- \otimes x_{i,0}^+z\right), \\
        \sum_{n\geq 0}\Theta_{\Psi_{j,a}, n\alpha_i}^-(z) &= \exp_{q_i}\left( \delta_{ij} (q_i-q_i^{-1}) ax_{i,0}^-\otimes x_{i,-1}^+z\right).
    \end{align*}
    More generally, let $\Bp$ be an $I$-tuple of Drinfeld polynomials such that $\Bp_j(u) = 1 - a u$. Factorizing $\Bp$ as a product of prefundamental $\ell$-weights and applying the multiplicativity property of Theta series in \cite[comments below (9.49)]{Z2}, we see that  $\Theta_{\Psi_{j,a}}^{\pm}(z)$ and $\Theta_{\Bp}^{\pm}(z)$ have the same $n\alpha_j$-components for $n \in \BN$. 
\end{example}

\section{Conjugation formulas for the quantum affine algebra}  \label{sec: conjugation formula qaf}
The two unitriangular R-matrices $\BR^{\pm}(z)$ are power series in $z$ with coefficients in the 0-completed tensor product $U_q^-(\Gaff) \ttimes U_q^+(\Gaff)$. In this section, we take suitable quotients of $U_q^{\pm}(\Gaff)$ and simplify the resulting power series as conjugations of Theta series by T-series. All finite-dimensional representations of the full quantum affine algebra $\qaf$ factorize through these quotients, so that our conjugation formulas are applicable and we prove the rationality of $\BR^{\pm}(z)$ evaluated at these representations.

\begin{defi}  \label{def: quotient qaf}
    For $\Bp = (\Bp_i(u))_{i\in I}$ an $I$-tuple of Drinfeld polynomials, we define $\mathbb{A}_{\Bp}^+$  to be the quotient of $U_q^+(\Gaff)$ by the two-sided ideal generated by the coefficients of $\Bp_i(u) x_i^+(u)$ for all $i \in I$. Similarly, define $\mathbb{A}_{\Bp}^-$ to be the quotient of $U_q^-(\Gaff)$ by the two-sided ideal generated by the coefficients of $\Bp_i(u)x_i^-(u)$ for $i \in I$. Let $ \pi_{\Bp}^{\pm}: U_q^{\pm}(\Gaff) \longrightarrow \mathbb{A}_{\Bp}^{\pm}$ denote the corresponding quotient morphisms of $\BQ$-graded algebras. 
\end{defi}
\begin{rem}
We conjecture $\mathbb{A}_{\Bp}^{\pm}$ to be finite-dimensional for any $I$-tuple $\Bp$ of Drinfeld polynomials. As an example,
   assume $\Bp_i(u) = 1 - a_i u$ with $a_i \in \BC^{\times}$ for each $i \in I$ and $a_i \neq a_j q^{b_{ij}}$ whenever $b_{ij} \neq 0$. Let $y_i \in \mathbb{A}_{\Bp}^+$ denote the image of $x_{i,0}^+$ under the quotient map $\pi_{\Bp}^+$. Then we have $\pi_{\Bp}(x_{i,n}^+) = a_i^n y_i$, so the algebra $\mathbb{A}_{\Bp}^+$ is generated by $y_i$ for $i\in I$. The Drinfeld relations and degree-two Serre relations are simplified as
   $$  (a_i-q^{b_{ij}} a_j) y_iy_j = (q^{b_{ij}}a_i-a_j) y_jy_i \quad \textrm{if $b_{ij} \neq 0$}, \quad y_i y_j = y_j y_i \quad \textrm{if $b_{ij} = 0$}. $$
Our assumption on $\Bp$ forces $y_iy_j \in \BC^{\times} y_jy_i$ and $y_i^2 = 0$ for all $i, j \in I$, so that the algebra $\mathbb{A}_{\Bp}^+$ is finite-dimensional. 
\end{rem}

 We arrive at the first main result of this paper. 
\begin{theorem}  \label{thm: conjugation formula qaf}
    Let $\Bp = (\Bp_i(u))_{i\in I}$ be an $I$-tuple of Drinfeld polynomials. Then we have the following conjugation formulas in $(U_q^-(\Gaff) \ttimes \mathbb{A}_{\Bp}^+)[[z]]$ for the first two equations and in $(\mathbb{A}_{\Bp}^- \ttimes U_q^+(\Gaff))[[z]]$ for the last two equations:
    \begin{align*}
        (\mathrm{Id} \otimes \pi_{\Bp}^+)(\BR^-(z)) &= \left( T_{\Bp}^+(z)\otimes 1\right)^{-1} \times (\tau_z \otimes \pi_{\Bp}^+)(\Theta_{\Bp}^+(1)) \times \left( T_{\Bp}^+(z)\otimes 1\right),  \\
        (\pi_{\Bp}^- \otimes \mathrm{Id})(\BR^+(z)) &= \left(1\otimes T_{\Bp}^+(z)\right) \times (\pi_{\Bp}^- \otimes \tau_z)(\Theta_{\Bp}^+(1)) \times \left(1\otimes T_{\Bp}^+(z)\right)^{-1}, \\
        (\pi_{\Bp}^- \otimes \mathrm{Id})(\BR^-(z)) &= \left(1\otimes T_{\Bp}^-(z)\right) \times (\pi_{\Bp}^- \otimes \tau_{z^{-1}})\left(\Theta_{\Bp}^-(1)^{-1}\right) \times \left(1\otimes T_{\Bp}^-(z)\right)^{-1},\\
        (\mathrm{Id} \otimes \pi_{\Bp}^+)(\BR^+(z)) &= \left( T_{\Bp}^-(z)\otimes 1\right)^{-1} \times (\tau_{z^{-1}} \otimes \pi_{\Bp}^+)\left(\Theta_{\Bp}^-(1)^{-1}\right) \times \left( T_{\Bp}^-(z)\otimes 1\right).
    \end{align*}
\end{theorem}
\begin{proof}
   We only prove the first conjugation formula, based on the first intertwining equation of Proposition \ref{prop: unitriangular R qaf}(2). The same idea applied to the last three intertwining equations gives the last three conjugation formulas.

\medskip

\noindent {\bf Step 1: Key formula.} According to \eqref{def: T} and \eqref{def: T gen qaf}, the power series $T_{\Bp}^+(z)$ is the exponential of $\sum_{s>0} h_{\Bp,s} z^s$ where each $h_{\Bp,s}$ is a linear combination of $h_{i,s}$ for $i \in I$. The first intertwining equation of Proposition \ref{prop: unitriangular R qaf}(2) remains true by replacing $h_{i,s}$ with $h_{\Bp,s}$. Multiply the resulting equation by $u^s$, sum over $s > 0$, and then take the exponential. We get the following equation in the algebra $(U_q^{\leq}(\Gaff)\ttimes U_q^{\geq}(\Gaff))((z))[[u]]$:
    $$ \left( T_{\Bp}^+(zu) \otimes T_{\Bp}^+(u) \right) \times \BR^-(z) = \BR^-(z) \times (\tau_z \otimes \mathrm{Id})\circ \Delta(T_{\Bp}^+(u)). $$
    All the first three factors belong to the subalgebra $(U_q^{\leq}(\Gaff)\ttimes U_q^{\geq}(\Gaff))[[z]][[u]]$. Since the third factor is constant in $u$,  the fourth factor belongs to this subalgebra. 
Replacing $\Delta(T_{\Bp}^+(u))$ by Theta series via Eq.\eqref{def: Theta}, and left multiplying by $1\otimes T_{\Bp}^+(u)^{-1}$, we obtain the {\bf key formula} in the algebra $(U_q^{\leq}(\Gaff)\ttimes U_q^\geq(\Gaff))[[u,z]]$:
$$ \left( T_{\Bp}^+(zu) \otimes 1\right) \times \BR^-(z) = (\mathrm{Id} \otimes \mathrm{Ad}_{T_{\Bp}^+(u)}^{-1})(\BR^-(z)) \times (\tau_z \otimes \mathrm{Id})(\Theta_{\Bp}^+(u)) \times (T_{\Bp}^+(zu) \otimes 1).  $$

\medskip

\noindent {\bf Step 2: polynomiality in $u$.}
We show that the five factors of the key formula are power series in $z$ with coefficients in the 0-completed tensor product $U_q^{\leq}(\Gaff) \ttimes (U_q^+(\Gaff)[u])$. Namely, the five factors lie in the following subalgebra of $(U_q^{\leq}(\Gaff)\ttimes U_q^\geq(\Gaff))[[u,z]]$: 
$$ \left(U_q^{\leq}(\Gaff) \ttimes (U_q^+(\Gaff)[u])\right)[[z]] =    \prod_{\beta \in \BQ_+} \left( U_q^{\leq}(\Gaff)_{-\beta} \otimes U_q^+(\Gaff)_{\beta}[u] \right)[[z]].  $$
\begin{itemize}
  \item[(i)] For the first, second and fifth factors, this is evident.
    \item[(ii)] The fourth factor is the sum, over $\beta \in \BQ_+$, of $(\tau_z \otimes \mathrm{Id})(\Theta_{\Bp,\beta}^+(u))$. Step 1 shows that no negative powers of $z$ appear in such a $\beta$-component. Together with Theorem \ref{thm: Theta qaf}, we see that the $\beta$-component lies in $U_q^-(\Gaff)_{-\beta} \otimes U_q^+(\Gaff)_{\beta}[u,z]$.
    \item[(iii)] The third factor is the sum, over $\beta \in \BQ_+$, of power series $(\mathrm{Id}\otimes \mathrm{Ad}_{T_{\Bp}^+(u)}^{-1})(\BR_{\beta}^-(z))$. By Proposition \ref{prop: unitriangular R qaf}(1), the second tensor factor of $\BR_{\beta}^-(z)$ lies in the weight space $U_q^+(\Gaff)_{\beta}$. By Eq.\eqref{rel: Tp x qaf}, conjugation by $T_{\Bp}^+(u)^{-1}$ sends each such weight space $U_q^+(\Gaff)_{\beta}$ to the space $U_q^+(\Gaff)_{\beta}[u]$ of polynomials. As a consequence,
    $$  (\mathrm{Id}\otimes \mathrm{Ad}_{T_{\Bp}^+(u)}^{-1})(\BR_{\beta}^-(z)) \in (U_q^-(\Gaff)_{-\beta} \otimes U_q^+(\Gaff)_{\beta}[u]) [[z]]. $$
\end{itemize}

\medskip

\noindent {\bf Step 3: evaluation at $u = 1$.}
In the key formula, let us evaluate $u$ at 1, left multiply both sides by $T_{\Bp}^+(z)^{-1} \otimes 1$, and apply $\mathrm{Id} \otimes \pi_{\Bp}^+$. This gives an equation
$$(\mathrm{Id}\otimes \pi_{\Bp}^+)(\BR^-(z)) = \mathrm{Ad}_{T_{\Bp}^+(z)\otimes 1}^{-1} \left(\lim_{u\rightarrow 1}\left(\mathrm{Id} \otimes \pi_{\Bp}^+ \circ \mathrm{Ad}_{T_{\Bp}^+(u)}^{-1})(\BR^-(z)\right) \times (\tau_z \otimes \pi_{\Bp}^+)(\Theta_{\Bp}^+(1)\right)$$
in the algebra $(U_q^-(\Gaff) \ttimes \mathbb{A}_{\Bp}^+)[[z]]$.
In comparison with the first conjugation formula, it suffices to prove that the limit factor at the right-hand side is $1\otimes 1$. Since the second tensor factor of $\BR^-(z)$ lies in $U_q^+(\Gaff)$, it is enough to show that
$$ \lim_{u\rightarrow 1} \pi_{\Bp}^+\left( T_{\Bp}^+(u)^{-1} x_i^+(z) T_{\Bp}^+(u)\right) = 0 \quad \textrm{for $i \in I$}. $$
By \eqref{rel: Tp x qaf}, the conjugation is given by $\Bp_i^*(z^{-1}u) x_i^+(z) \in U_q^+(\Gaff)[u][[z,z^{-1}]]$. Its evaluation at $u=1$ is the formal Laurent series $\Bp_i^*(z^{-1}) x_i^+(z)$, whose coefficients belong to the kernel of $\pi_{\Bp}^+$ because $\Bp_i^*(z^{-1}) = \lambda z^k \Bp_i(z)$ for certain $\lambda \in \BC^{\times}$ and $k \in \BZ$. 
\end{proof}
\begin{rem}
    Let us use the key formula at Step 1 to give a direct proof of Theorem \ref{thm: Theta qaf}. It suffices to show that each $(\tau_z \otimes \mathrm{Id})(\Theta_{\Bp,\beta}(u))$, for $\beta \in \BQ_+$, is a polynomial in $u$. For $\gamma \in \BQ$, let $\BR_{\gamma}^{*}(z)$ denote the $(-\gamma,\gamma)$-component of the unitriangular power series $\BR^-(z)^{-1} \in (U_q^-(\Gaff)\ttimes U_q^+(\Gaff))[[z]]$. Then the key formula implies a finite summation
    $$ (\tau_z \otimes \mathrm{Id})(\Theta_{\Bp,\beta}^+(u)) = \sum_{\gamma\in \BQ_+} (\mathrm{Id} \otimes \mathrm{Ad}_{T_{\Bp}^+(u)}^{-1})(\BR_{\beta-\gamma}^{*}(z)) \times (\mathrm{Ad}_{T_{\Bp}^+(zu)} \otimes \mathrm{Id})(\BR_{\gamma}^-(z)).  $$
   We argue as in Step 2(iii) to obtain the desired polynomiality in $u$.
\end{rem}
To apply the theorem, we need representations of the quotient algebras $\mathbb{A}_{\Bp}^{\pm}$.
 The next result is well-known and dates back to \cite[\S 6]{BeK}.  Given a formal Laurent series $f(u) = \sum_{n\in \BZ}f_nu^n$ and $k \in \BZ$, let $f(u)_{\geq k}$ denote $\sum_{n\geq k}f_n u^n$, which is a Laurent series in $u$. Similarly $f(u)_{<k}$ is defined as a Laurent series in $u^{-1}$.
\begin{lem}  \label{lem: quotient qaf} 
Let $\varphi: U_q^{\geq}(\Gaff) \longrightarrow \mathrm{End}(W)$ be a finite-dimensional representation.
\begin{itemize}
    \item[(1)] Let $i \in I,\, k \in \BZ$ and $f(u) \in \BC[u]$ be a Drinfeld polynomial. Then the following two vector-valued formal Laurent series are Laurent expansions at $u=0$ and $u=\infty$ respectively of the same rational function of $u$: 
    $$ \varphi(x_i^+(u)_{\geq k}) \in \mathrm{End}(W)((u)), \quad -\varphi(x_i^+(u)_{<k}) \in \mathrm{End}(W)((u^{-1})). $$
    We have $f(u) \varphi(x_i^+(u)) = 0$ if and only if $f(u) \varphi(x_i^+(u)_{\geq k})$ is a Laurent polynomial in $u$, if and only if $f(u) \varphi(x_i^+(u)_{<k})$ is a Laurent polynomial in $u$. 
    \item[(2)] There exists an $I$-tuple of Drinfeld polynomials $\Bp$ such that the restriction of $\varphi$ to $U_q^+(\Gaff)$ factorizes through the quotient map $\pi_{\Bp}^+$.
\end{itemize}
Similar statements hold true for finite-dimensional representations of $U_q^{\leq}(\Gaff)$.
\end{lem}

Given a representation $\varphi: U_q^+(\Gaff) \longrightarrow \mathrm{End}(W)$, let us make sense of the coefficients of the power series $(\mathrm{Id}\otimes \varphi)(\BR^{\pm}(z))$. For $\beta \in \BQ$, we define 
$$\mathbb{E}_{W,\beta} := \{ f \in \mathrm{End}(W)\ |\ \rho(K_i) \circ f = q^{(\alpha_i,\beta)} f \circ \rho(K_i) \quad \textrm{for $i \in I$}   \}. $$
Clearly, $\varphi$ sends $U_q^+(\Gaff)_{\beta}$ to $\mathbb{E}_{W,\beta}$. The sum of the subspaces $\mathbb{E}_{W,\beta}$ for $\beta \in \BQ$ is direct and forms a subalgebra of $\mathrm{End}(W)$ that contains the image of $\varphi$. Let $\mathbb{E}_W$ denote the resulting $\BQ$-graded algebra, so that $(\mathrm{Id} \otimes \varphi)(\BR^{\pm}(z))\in (U_q^-(\Gaff) \ttimes \mathbb{E}_W)[[z]]$. If the restriction of the representation $\varphi$  to $U_q^+(\Gaff)$ factorizes through certain $\pi_{\Bp}^+$, then we obtain from Theorem \ref{thm: conjugation formula qaf} two equations in $(U_q^-(\Gaff)\ttimes \mathbb{E}_W)[[z]]$:
\begin{align*}
        (\mathrm{Id} \otimes \varphi)(\BR^-(z)) &= \left(\mathrm{Ad}_{T_{\Bp}^+(z)}^{-1} \circ \tau_z\otimes  \varphi\right)(\Theta_{\Bp}^+(1)), \\
        (\mathrm{Id} \otimes \varphi)(\BR^+(z)) &= \left(\mathrm{Ad}_{T_{\Bp}^-(z)}^{-1} \circ \tau_{z^{-1}}\otimes \varphi\right)\left(\Theta_{\Bp}^-(1)^{-1}\right).
    \end{align*}

\begin{prop}   \label{prop: rationality qaf}
    Let $\varphi: \qaf \longrightarrow \mathrm{End}(V)$ and $\psi: \qaf \longrightarrow \mathrm{End}(W)$ be two finite-dimensional representations. Then both power series $(\varphi \otimes \psi)(\BR^{\pm}(z))$ are Taylor expansions at $z=0$ of $\mathrm{End}(V\otimes W)$-valued rational functions of $z$. Any pole of the rational functions $(\varphi\otimes \psi)(\BR^-(z))$ and $(\varphi\otimes \psi)(\BR^+(z^{-1}))$ is necessarily the ratio of a pole of the rational function $\varphi(x_i^-(z)_{>0})$ to a pole of $\psi(x_i^+(z)_{\geq 0})$ for certain $i \in I$. 
\end{prop}
\begin{proof}
    We consider the case of $\BR^-(z)$, as the other case is parallel. By Lemma \ref{lem: quotient qaf}, the restriction of $\psi$ to $U_q^+(\Gaff)$ factorizes through $\pi_{\Bp}^+$ for an $I$-tuple $\Bp = (\Bp_i(u))_{i\in I}$ of Drinfeld polynomials; we take $\Bp_i(u)$ to be the denominator of the vector-valued rational function $\psi(x_i^+(u)_{\geq 0})$.  Then by the first conjugation formula, 
    $$ (\varphi\otimes \psi)(\BR^-(z)) = (\varphi \circ \mathrm{Ad}_{T_{\Bp}^+(z)}^{-1} \circ \tau_z \otimes \psi)(\Theta_{\Bp}^+(1)).   $$
    Since $\psi$ is finite-dimensional, by Theorem \ref{thm: Theta qaf} the right-hand side lies in the subalgebra of $\mathrm{End}(V\otimes W)[[z]]$ generated by the following vectors 
    $$ \varphi \circ \mathrm{Ad}_{T_{\Bp}^+(z)}^{-1} (x_{i,k}^-z^k) \otimes \psi(x_{j,n}^+) \quad \textrm{for $i,j \in I$ and $k,\, n \in \BN$}.  $$
    It is therefore enough to prove the rationality of $\varphi\circ \mathrm{Ad}_{T_{\Bp}^+(z)}^{-1} (x_{i,k}^-z^k)$ for all $i \in I$ and $k \in \BN$. Let $\sigma$ denote the algebra automorphism of $U_q^-(\Gaff)$ which sends $x_{j,m}^-$ to $x_{j,m+1}^-$ for $j \in I$ and $m \in \BZ$. Then by Eq.\eqref{rel: Tp x qaf} we have
    $$ T_{\Bp}^+(z)^{-1} x_{i,k}^- T_{\Bp}^+(z) = \frac{1}{\Bp_i^*(\sigma z)}(x_{i,k}^-). $$
    If $\Bp_i(u) = 1$, then the rationality is clear. Assume $\Bp_i(u) \neq 1$  and take the partial fraction decomposition 
    $$ \frac{1}{\Bp_i^*(z)} = \sum_{(b,n)\in \Lambda} \frac{\lambda_{b,n}}{(1-bz)^{n+1}} $$
    where $\Lambda$ is a finite subset of $\BC^{\times} \times \BN$, and $\lambda_{b,n} \in \BC^{\times}$ for any $(b,n) \in \Lambda$. Each such $b$ is a pole of the rational function $\psi(x_i^+(z)_{\geq 0})$. We have 
    \begin{align*}
        \frac{1}{(1-b\sigma z)^{n+1}}(x_{i,k}^-z^k) &= \frac{\sigma^{-n}}{n! b^n} \partial_z^n\left(\frac{1}{1-b\sigma z}(x_{i,k}^-z^k)\right) = \frac{\partial_z^n}{n!}\left(b^{-k}z^n \sum_{m\geq k-n}x_{i,m}^-(bz)^m\right), \\
        \mathrm{Ad}_{T_{\Bp}^+(z)}^{-1} (x_{i,k}^-z^k) &= \sum_{(b,n)\in \Lambda} \frac{\lambda_{b,n}}{n! b^k} \partial_z^n\left(z^n x_i^-(bz)_{\geq k-n} \right).
    \end{align*}
    The rationality follows from the negative counterpart of Lemma \ref{lem: quotient qaf}(1).
\end{proof}

The next result is essentially due to Hernandez \cite[Proposition 4.4]{H1}.

\begin{prop}  \label{prop: rationality qaf inverse}
    In the situation of Proposition \ref{prop: rationality qaf}, the two $\mathrm{End}(V\otimes W)$-valued rational functions $(\varphi\otimes \psi)(\BR^-(z))$ and $(\varphi\otimes \psi)(\BR^+(z^{-1}))$ are inverses of each other.
\end{prop}
\begin{proof}
    By \cite[Proposition 3.8]{MY}, each finite-dimensional representation is a direct sum of submodules with the following property as a replacement of \cite[Proposition 2.14]{H1}: the ratio of two $\ell$-weights of this submodule is always a Laurent monomial in the generalized simple roots. The proofs of \cite[Propositions 3.8 \& 4.4]{H1} work for these submodules.
\end{proof}

\begin{rem}
  (1) When the two representations $\varphi$ and $\psi$ are irreducible, the rationality of $(\varphi\otimes \psi)(\BR^{\pm}(z))$ was established in \cite[Theorem 3.9, Proposition 4.4]{H1} by a different method. Our approach is more direct and drops the irreducibility assumption. 
 
  (2) Each representation in the category $\mathcal{O}$ for the quantum affine algebra \cite[\S 4.3]{H0} is naturally a union of finite-dimensional $U_q^+(\Gaff)$-modules, so the conjugation formulas of Theorem \ref{thm: conjugation formula qaf} are still applicable. Propositions \ref{prop: rationality qaf}--\ref{prop: rationality qaf inverse} hold true for arbitrary representations $\varphi$ and $\psi$ in this category with the same proofs. 

  (3) Theorem \ref{thm: conjugation formula qaf} and Proposition \ref{prop: rationality qaf} can be applied to the category $\mathcal{O}^{sh}$ of representations of shifted quantum affine algebras \cite[\S 4.4]{H2} for the same reason as (2). We expect Proposition \ref{prop: rationality qaf inverse} to be true in the shifted case. The main obstacle is the lack of Drinfeld--Jimbo coproduct for shifted quantum affine algebras to make sense of the triangularity of $\ell$-weight vectors in \cite[Theorem 2.16]{H1}.
\end{rem}

  \begin{example}  \label{ex: Ding-Frenkel}
        Let $\Glie = \mathfrak{sl}_{r+1}$. On the vector space $\BC^{r+1}$ there is a representation $\varphi$ of the quantum affine algebra $\qaf$ given by:
        $$ \varphi(x_{i,n}^+) = q^{in} E_{i,i+1},\quad \varphi(x_{i,n}^-) = q^{in} E_{i+1,i} \quad \textrm{for $1\leq i \leq r$,} $$
        where the $E_{ij} \in \mathrm{End}(\BC^{r+1}) = \mathrm{Mat}_{r+1}(\BC)$ are elementary matrices.
        This is an representation of highest $\ell$-weight  $Y_{i,1}$ in \cite[\S 5.4.1]{FR1}. Its restrictions to $U_q^{\pm}(\Gaff)$ factorize through $\pi_{\Bp}^{\pm}$ for the $I$-tuple $\Bp = (1-q^i u)_{1\leq i \leq r}$. From Example \ref{ex: Theta qaf} we get
        \begin{align*}
        &(\mathrm{Id}\otimes \varphi)(\BR_{\alpha_i}^-(z)) = (\mathrm{Ad}_{T_{\Bp}^+(z)}^{-1} \circ \tau_z \otimes \varphi)(\Theta_{\Bp,\alpha_i}^+(1)) \\
        &\qquad = (q^{-1}-q)q^{-i} z \mathrm{Ad}_{T_{\Bp}^+(z)}^{-1}( x_{i,1}^-) \otimes \varphi(x_{i,0}^+) = (q^{-1}-q)  \sum_{n\geq 1} x_{i,n}^-(q^{-i}z)^n \otimes E_{i,i+1},    \\
            &(\varphi \otimes \mathrm{Id})(\BR_{\alpha_i}^+(z)) = (\varphi \otimes \mathrm{Ad}_{T_{\Bp}^+(z)} \circ \tau_z)(\Theta_{\Bp,\alpha_i}^+(1))  \\
            &\qquad = (q^{-1}-q) q^{-i} \varphi(x_{i,1}^-) \otimes \mathrm{Ad}_{T_{\Bp}^+(z)}(x_{i,0}^+) = (q^{-1}-q) E_{i+1,i} \otimes \sum_{n\geq 0} x_{i,n}^+ (q^{-i}z)^n,  \\
            &(\varphi \otimes \mathrm{Id})(\BR_{\alpha_i}^-(z)) = (\varphi \otimes \mathrm{Ad}_{T_{\Bp}^-(z)} \circ \tau_{z^{-1}})(-\Theta_{\Bp,\alpha_i}^-(1)) \\
            &\qquad = \varphi(x_{i,0}^-) \otimes (q^{-1}-q) q^i z \mathrm{Ad}_{T_{\Bp}^-(z)}(x_{i,-1}^+) = (q^{-1}-q) E_{i+1,i} \otimes \sum_{n<0} x_{i,n}^+ (q^{-i} z^{-1})^n, \\
            & (\mathrm{Id} \otimes \varphi)(\BR_{\alpha_i}^+(z)) = (\mathrm{Ad}_{T_{\Bp}^-(z)}^{-1} \circ \tau_{z^{-1}} \otimes \varphi)(-\Theta_{\Bp,\alpha_i}^-(1)) \\
            &\qquad = (q^{-1}-q) q^i\mathrm{Ad}_{T_{\Bp}^-(z)}^{-1}(x_{i,0}^-) \otimes \varphi(x_{i,-1}^+) = (q^{-1}-q) \sum_{n\leq 0} x_{i,n}^- (q^{-i} z^{-1})^n \otimes E_{i,i+1}.
        \end{align*}
        The generating series $x_i^{\pm}(q^{-i}z)$ are recovered as off-diagonal entries of the unitriangular L-operators associated with $\varphi$.  This agrees with the Ding--Frenkel \cite[(3.23)]{DF} homomorphism from the Drinfeld new realization of the quantum affien algebra to its R-matrix realization. We comment that a direct computation of these off-diagonal entries without Theta series appeared in \cite[Proof of Lemma 3.6]{FM}.
    \end{example}

\section{Background on shifted Yangians}  \label{sec: shifted Yangians}
The goal of the rest of this paper is to adapt the conjugation formulas of Theorem \ref{thm: conjugation formula qaf} to the Yangian situation. For that purpose, we shall need the notion of shifted Yangians. In this section, we collect the basic properties of shifted Yangians and their representation theory.

Recall the finite-dimensional simple Lie algebra $\Glie$, its Cartan subalgebra $\Hlie$, and the set $I$ of Dynkin nodes. The dual Cartan subalgebra $\Hlie^*$ admits a basis consisting of simple roots $\alpha_i$ for $i \in I$.
Let $(\varpi_i^{\vee})_{i \in I}$ be the dual basis of $\Hlie$  with respect to the natural pairing $\langle, \rangle: \Hlie \times \Hlie^* \longrightarrow \BC$; the $\varpi_i^{\vee}$ are called {\it fundamental coweights}. The fundamental coweights generate an additive subgroup of $\Hlie$, called the coweight lattice, and denoted by $\BP^{\vee}$.
By a coweight, we mean an element $\mu$ of the coweight lattice. It is a $\BZ$-linear combination $\sum_{i\in I} m_i \varpi_i^{\vee}$ of the fundamental coweights with coefficients $m_i = \langle \mu, \alpha_i \rangle$. Call $\mu$ {\it dominant} if $m_i \in \BN$ for all $i \in I$; call $\mu$ {\it antidominant} if $-\mu$ is dominant. 

\subsection{Shifted Yangians} 
Given a coweight $\mu = \sum_{i\in I} m_i \varpi_i^{\vee}$, we define the shifted Yangian $Y_{\mu}(\Glie)$ to be the associative algebra with generators
 $$ x_{i,n}^{\pm},\quad \xi_{i,p} \quad \mathrm{for}\ (i, n, p) \in I \times \BN \times \BZ $$
 called Drinfeld generators, subject to the following relations \cite{Dri88,KWWY,BFN}: 
 \begin{gather*}
[\xi_{i,p}, \xi_{j,q}] = 0, \quad [x_{i,m}^+,x_{j,n}^-] = \delta_{ij} \xi_{i,m+n},   \\
[\xi_{i,p+1}, x_{j,n}^{\pm}] - [\xi_{i,p}, x_{j,n+1}^{\pm}] = \pm \frac{1}{2} b_{ij} (\xi_{i,p}x_{j,n}^{\pm} + x_{j,n}^{\pm} \xi_{i,p}),\\
[x_{i,m+1}^{\pm}, x_{j,n}^{\pm}] - [x_{i,m}^{\pm}, x_{j,n+1}^{\pm}] = \pm \frac{1}{2} b_{ij} (x_{i,m}^{\pm}x_{j,n}^{\pm} + x_{j,n}^{\pm} x_{i,m}^{\pm}), \\
\mathrm{ad}_{x_{i,0}^{\pm}}^{1-c_{ij}}(x_{j,0}^{\pm}) = 0\quad \mathrm{if}\ i \neq j,  \\
\xi_{i,-m_i- 1} = 1,\quad \xi_{i,p} = 0 \quad \mathrm{for}\ p < -m_i - 1. 
\end{gather*}
Here $\mathrm{ad}_x(y) := xy - yx$.
We define four families of generating series in a formal variable $u$, all indexed by $i \in I$:
 \begin{equation*} 
x_i^{\pm}(u) :=  \sum_{n\in \BN} x_{i,n}^{\pm} u^{-n-1},\quad \xi_i(u) :=  \sum_{p\in \BZ} \xi_{i,p} u^{-p-1},\quad \overline{\xi}_i(u) := u^{-m_i} \xi_i(u).
\end{equation*}
These are Laurent series in $Y_{\mu}(\Glie)((u^{-1}))$, with leading terms $x_{i,0}^{\pm} u^{-1},\ u^{m_i}$ and 1.

The shifted Yangian $Y_{\mu}(\Glie)$ admits a $\BQ$-grading, called its {\it weight grading}, defined by declaring the weights of the generators $x_{i,n}^+, x_{i,n}^-$ and $\xi_{i,p}$ to be $\alpha_i, -\alpha_i$ and $0$. Equivalently, for $\beta \in \BQ$, an element $x \in Y_{\mu}(\Glie)$ is of weight $\beta$ if and only if $[\xi_{i,-m_i}, x] = (\alpha_i,\beta) x$ for all $i \in I$.
Let $Y_{\mu}(\Glie)_{\beta}$ denote the subspace of elements of weight $\beta$. 

Let $z$ be another formal variable, referred to as {\it spectral parameter}. We have an algebra homomorphism $\tau_z: Y_{\mu}(\Glie) \longrightarrow Y_{\mu}(\Glie)[z]$ defined by
\begin{gather}  \label{rel: spectral shift formal Yangian}
    \tau_z(X_p) = \sum_{n\in \BN} \binom{p}{n}  X_{p-n} z^n \quad \mathrm{for}\ X \in \{x_i^{\pm}, \xi_i\}\ \mathrm{and}\ p \in \BZ.
\end{gather}
Here it is understood that $x_{i,p}^{\pm} = 0$ for $p < 0$. In terms of generating series, we have 
$$ \tau_z(x_i^{\pm}(u)) = x_i^{\pm}(u-z), \quad \tau_z(\xi_i(u)) = \xi_i(u-z). $$
Evaluating $z$ at complex numbers, we get a one-parameter family of algebra automorphisms $\tau_a$ of $Y_{\mu}(\Glie)$ satisfying $\tau_a \circ \tau_b = \tau_{a+b}$ and $\tau_0 = \mathrm{Id}$ for $a, b \in \BC$. These are called {\it spectral parameter automorphisms}. 

In the shifted Yangian $Y_{\mu}(\Glie)$ let us define five subalgebras by generating subsets:
\begin{gather*}
Y_{\mu}^+(\Glie) = \langle x_{i,n}^+\rangle_{(i,n)\in I\times \BN},\quad Y_{\mu}^0(\Glie) = \langle \xi_{i,p}\rangle_{(i,p)\in I \times \BZ},\quad Y_{\mu}^-(\Glie) =  \langle x_{i,n}^-\rangle_{(i,n)\in I\times \BN}, \\
Y_{\mu}^{\geq}(\Glie) =  \langle x_{i,n}^+, \xi_{i,p}\rangle_{(i,n,p)\in I\times \BN \times \BZ},\quad Y_{\mu}^{\leq}(\Glie) = \langle x_{i,n}^-, \xi_{i,p}\rangle_{(i,n,p)\in I\times \BN \times \BZ}.
\end{gather*}
 The weight grading and the spectral parameter automorphisms restrict to these five subalgebras.
By definition, the subalgebra $Y^0_{\mu}(\Glie)$ is commutative.

\subsection{Drinfeld--Jimbo coproduct and its shifts}  \label{sss: Yangian coproduct}
In this subsection, we recall two families of algebra homomorphisms relating various shifted Yangians: the shift homomorphisms and the shifted coproducts.  

 The zero-shifted Yangian $Y_0(\Glie)$, also denoted by $Y(\Glie)$, is precisely the ordinary Yangian in its Drinfeld new realization \cite{Dri88} with deformation parameter $\hbar = 1$. Let us first recall its Hopf algebra structure. The Yangian contains the universal enveloping algebra $U(\Glie)$ as a Hopf subalgebra by identifying the $x_{i,0}^{\pm}$ with root vectors in the Lie algebra $\Glie$ associated to the roots $\pm \alpha_i$. Let $\Phi \subset \BQ_+$ denote the set of positive roots of $\Glie$. One can extend $x_{i,0}^{\pm} =: x_{\alpha_i}^{\pm}$ to root vectors $x_{\beta}^{\pm} \in \Glie_{\pm \beta}$ for $\beta \in \Phi$ suitably normalized with respect to an invariant bilinear form of $\Glie$. Then we have the following coproduct formulas (see \cite[\S 4]{GNW} for a proof):
    \begin{gather}   
\Delta(x_{i,0}^{\pm}) = x_{i,0}^{\pm} \otimes 1 \otimes x_{i,0}^{\pm}, \quad \Delta(\xi_{i,0}) = \xi_{i,0} \otimes 1 + 1 \otimes \xi_{i,0}, \\
        \Delta(\xi_{i,1}) = \xi_{i,1} \otimes 1 + 1 \otimes \xi_{i,1} + \xi_{i,0} \otimes \xi_{i,0} - \sum_{\beta \in \Phi} (\alpha_i, \beta) x_{\beta}^- \otimes x_{\beta}^+.  \label{equ: coproduct Yangian xi}
    \end{gather}
We refer to $\Delta$ as the Drinfeld--Jimbo coproduct, to distinguish it from the Drinfeld formal coproduct introduced in \cite{GTLW}. 

For antidominant coweights $\epsilon$ and $\eta$, we have an injective algebra morphism $\iota_{\epsilon,\eta}^{\mu}: Y_{\mu}(\Glie) \longrightarrow Y_{\mu+\epsilon+\eta}(\Glie)$, called {\it shift homomorphism} \cite[Corollary 3.16]{coproduct}:
\begin{gather}   \label{def: shift homo}
x_{i,n}^+ \mapsto x_{i,n-\langle\epsilon,\alpha_i\rangle}^+, \quad x_{i,n}^- \mapsto x_{i,n-\langle\eta,\alpha_i\rangle}^-,\quad \xi_{i,p} \mapsto \xi_{i,p-\langle\epsilon+\eta,\alpha_i\rangle}.
\end{gather}
It preserves the normalized series $\overline{\xi}_i(z) \mapsto \overline{\xi}_i(z)$. 
It also induces canonical identifications of subalgebras of shifted Yangians:
$$ Y_{\mu}^{\pm}(\Glie) \cong Y_{\nu}^{\pm}(\Glie),\qquad x_{i,n}^{\pm} \longrightarrow x_{i,n}^{\pm}. $$

The Drinfeld--Jimbo coproduct of the ordinary Yangian can be extended to shifted Yangians in a compatible way with the above shift homomorphisms.

\begin{theorem}\cite[Theorem 4.12, Proposition 4.14]{coproduct}  \label{thm: coproduct shifted Yangian}
There exists a unique family of algebra homomorphisms for all coweights $\mu, \nu$ $$\Delta_{\mu,\nu}: Y_{\mu+\nu}(\Glie) \longrightarrow Y_{\mu}(\Glie) \otimes Y_{\nu}(\Glie)$$ 
that satisfy the following properties: 
 \begin{itemize}
     \item[(1)] The zero-shifted case $\Delta_{0,0}$ is the coproduct of the ordinary Yangian.
     \item[(2)] If $0\leq m < - \langle \mu, \alpha_i\rangle$  and $0 \leq n < -\langle \nu,\alpha_i\rangle$, then 
     $$\Delta_{\mu,\nu}(x_{i,m}^+) = x_{i,m}^+ \otimes 1,\quad \Delta_{\mu,\nu}(x_{i,n}^-) = 1 \otimes x_{i,n}^-.  $$
     \item[(3)] It is compatible with shift homomorphisms for antidominant coweights $\epsilon$ and $\eta$:  
     $$ (\iota_{\epsilon,0}^{\mu} \otimes \iota_{0,\eta}^{\nu}) \circ \Delta_{\mu,\nu} = \Delta_{\mu+\epsilon,\nu+\eta} \circ \iota_{\epsilon,\eta}^{\mu+\nu}: Y_{\mu+\nu}(\Glie) \longrightarrow Y_{\mu+\epsilon}(\Glie) \otimes Y_{\nu+\eta}(\Glie). $$
     \item[(4)] It is coassociative if the middle coweight $\epsilon$ is antidominant:
     $$ (\Delta_{\mu,\epsilon} \otimes \mathrm{Id}) \circ \Delta_{\mu+\epsilon,\nu} = (\mathrm{Id} \otimes \Delta_{\epsilon,\nu}) \circ \Delta_{\mu,\epsilon+\nu}: Y_{\mu+\epsilon+\nu}(\Glie) \longrightarrow Y_{\mu}(\Glie)\otimes Y_{\epsilon}(\Glie) \otimes Y_{\nu}(\Glie). $$
 \end{itemize}
\end{theorem}
We call $\Delta_{\mu,\nu}$ a shifted coproduct, although there is no Hopf algebra structure involved. Given a $Y_{\mu}(\Glie)$-module $M$ and a $Y_{\nu}(\Glie)$-module $N$, we can now equip the tensor product $M\otimes N$ with a $Y_{\mu+\nu}(\Glie)$-module structure via $\Delta_{\mu,\nu}$.

The following coproduct estimation for the ordinary Yangian is due to Knight \cite{Kn}. Its proof works for shifted Yangians.
\begin{lem}\cite[Lemma 2.5]{HZ} \label{lem: Yangian coproduct estimation}
Let $\mu$ and $\nu$ be two coweights. For $i \in I,\ n \in \BN$ and $p \in \BZ$, we have the following coproduct estimation:
\begin{gather*}
\Delta_{\mu,\nu}(x_{i,n}^+) \equiv  x_{i,n}^+ \otimes 1 + \sum_{m\in \BN} \xi_{i,n-m-1} \otimes x_{i,m}^+  \ \mathrm{mod}.\sum_{0\neq \beta \in \BQ_+} Y_{\mu}^{\leq}(\Glie)_{-\beta}  \otimes Y_{\nu}^{\geq}(\Glie)_{\beta+\alpha_i},  \\
\Delta_{\mu,\nu}(x_{i,n}^-) \equiv 1 \otimes x_{i,n}^- + \sum_{m\in \BN} x_{i,m}^- \otimes \xi_{i,n-m-1}  \ \mathrm{mod}. \sum_{0\neq\beta \in \BQ_+}  Y_{\mu}^{\leq}(\Glie)_{-\beta-\alpha_i} \otimes  Y_{\nu}^{\geq}(\Glie)_{\beta}, \\
\Delta_{\mu,\nu}(\xi_{i,p}) \equiv \sum_{t\in \BZ} \xi_{i,t} \otimes \xi_{i,p-t-1} \ \mathrm{mod}. \sum_{0\neq \beta \in \BQ_+} Y_{\mu}^{\leq}(\Glie)_{-\beta} \otimes Y_{\nu}^{\geq}(\Glie)_{\beta}, \\
\Delta_{\mu,\nu}(\xi_{i,-\langle\mu+\nu,\alpha_i\rangle}) = \xi_{i,-\langle\mu,\alpha_i\rangle} \otimes 1 + 1 \otimes \xi_{i,-\langle\nu,\alpha_i\rangle}.
\end{gather*}
In particular, we have the co-ideal subalgebra property:
$$ \Delta_{\mu,\nu}(Y_{\mu+\nu}^{\geq}(\Glie)) \subset Y_{\mu}(\Glie) \otimes Y_{\nu}^{\geq}(\Glie),\quad  \Delta_{\mu,\nu}(Y_{\mu+\nu}^{\leq}(\Glie)) \subset Y_{\mu}^{\leq}(\Glie) \otimes Y_{\nu}(\Glie). $$
\end{lem} 
As a consequence, if $M$ is a $Y_{\mu}(\Glie)$-module and $N$ is a $Y_{\nu}^{\geq}(\Glie)$-module, then $M\otimes N$ is naturally a $Y_{\mu+\nu}^{\geq}(\Glie)$-module. Similar statement holds for $Y_{\mu}^{\leq}(\Glie)$-modules.

Another important application of the coproduct estimation is the exact projection formulas \eqref{equ: trivial times}--\eqref{equ: times  trivial} below. 

\begin{defi}\cite[Definition 3.1]{Z2}   \label{defi: trivial module}
    For $\mu$ a coweight, let $\BY_{\mu}(\Glie)$ denote the quotient of $Y_{\mu}(\Glie)$ by the two-sided ideal generated by $x_{i,m}^{\pm}$ for $i \in I$ and $m \in \BN$. Denote the quotient map by $\pi_{\mu}: Y_{\mu}(\Glie) \longrightarrow \BY_{\mu}(\Glie)$. Call a $Y_{\mu}(\Glie)$-module {\it trivial} if its structural map factorizes through the quotient map $\pi_{\mu}$, namely, the actions of the $x_{i,m}^{\pm}$ are identically zero. Similarly, call a $Y_{\mu}^{\geq}(\Glie)$-module {\it trivial} if the actions of the $x_{i,m}^+$ are identically zero. Call a $Y_{\mu}^{\leq}(\Glie)$-module {\it trivial} if the actions of the $x_{i,m}^-$ are identically zero.
\end{defi}

By abuse of language, let $\xi_i(u)$ denote the image of $\xi_i(u)$ by $\pi_{\mu}$. Consider the algebra homomorphism $(\pi_{\mu} \otimes \mathrm{Id}) \circ \Delta_{\mu,\nu}: Y_{\mu+\nu}(\Glie) \longrightarrow \BY_{\mu}(\Glie) \otimes Y_{\nu}(\Glie)$, the images by which of the generating series of $Y_{\mu+\nu}(\Glie)$ are given by \cite[Eq.(3.6)]{Z2}
\begin{equation}  \label{equ: trivial times}
    x_i^+(u) \mapsto  \langle \xi_i(u) \otimes x_i^+(u)\rangle_+, \quad x_i^-(u) \mapsto 1 \otimes x_i^-(u),\quad \xi_i(u) \mapsto \xi_i(u) \otimes \xi_i(u).
\end{equation}
Here, for a formal Laurent series $f = \sum_{p\in \BZ} f_p u^{-p-1}$ with coefficients in a vector space, the symbol $\langle f\rangle_+$ means the power series $\sum_{n\in \BN} f_n u^{-n-1}$; sometimes, we write $\langle f\rangle_+^u$ to emphasize the variable $u$. Similarly, the algebra homomorphism $(\mathrm{Id} \otimes \pi_{\nu}) \circ \Delta_{\mu,\nu}$ from $Y_{\mu+\nu}(\Glie)$ to $Y_{\mu}(\Glie) \otimes \BY_{\nu}(\Glie)$ is described in terms of the generating series as:
\begin{equation}  \label{equ: times  trivial}
    x_i^+(u) \mapsto x_i^+(u) \otimes 1, \quad x_i^-(u) \mapsto \langle x_i^-(u) \otimes \xi_i(u) \rangle_+,\quad \xi_i(u) \mapsto \xi_i(u) \otimes \xi_i(u).
\end{equation}
\begin{rem}  \label{rem: coideal}
    The formulas \eqref{equ: trivial times}--\eqref{equ: times  trivial} can be restricted to coideal subalgebras. Let $\varphi: Y_{\mu}(\Glie) \longrightarrow \mathrm{End}(M)$ and $\psi: Y_{\nu}^{\geq}(\Glie) \longrightarrow \mathrm{End}(N)$ be two representations, one of which is trivial. Then in the $Y_{\mu+\nu}^{\geq}(\Glie)$-module $M\otimes N$, each Drinfeld--Cartan series $\xi_i(u)$ acts simply as $\varphi(\xi_i(u)) \otimes \psi(\xi_i(u))$. Similar statements hold for $Y_{\mu+\nu}^{\leq}(\Glie)$. 
\end{rem}

\subsection{Co-associativity and Theta series}   \label{ss: Theta}
Given three coweights $\mu, \epsilon$ and $\nu$, consider the two sides of the equation in Theorem \ref{thm: coproduct shifted Yangian}(4):
$$ \Delta_{(\mu\epsilon)\nu} := (\Delta_{\mu,\epsilon} \otimes \mathrm{Id}) \circ \Delta_{\mu+\epsilon,\nu}\quad \textrm{and}\quad \Delta_{\mu(\epsilon\nu)} := (\mathrm{Id}\otimes \Delta_{\epsilon,\nu}) \circ \Delta_{\mu,\epsilon+\nu} $$
as algebra homomorphisms from $Y_{\mu+\epsilon+\nu}(\Glie)$ to $Y_{\mu}(\Glie) \otimes Y_{\epsilon}(\Glie) \otimes Y_{\nu}(\Glie)$. Let $M,\, K$ and $N$ be modules over the corresponding three shifted Yangians. On the same triple tensor product $M \otimes K \otimes N$, we have two module structures over $Y_{\mu+\epsilon+\nu}(\Glie)$ induced by $\Delta_{(\mu\epsilon)\nu}$ and $\Delta_{\mu(\epsilon\nu)}$, and denoted by $(M\otimes K) \otimes N$ and $M \otimes (K \otimes N)$, respectively. In this subsection, we recall the results of \cite{Z2} on comparisons of these two modules.

\begin{theorem}\cite[Theorem 3.3]{Z2}   \label{thm: asso trivial}
    Given three coweights $\mu, \epsilon$ and $\nu$, we have the following two identities of algebra homomorphisms from $Y_{\mu+\epsilon+\nu}(\Glie)$ to $Y_{\mu}(\Glie)\otimes Y_{\epsilon}(\Glie) \otimes \BY_{\nu}(\Glie)$ and to $\BY_{\mu}(\Glie) \otimes Y_{\epsilon}(\Glie) \otimes Y_{\nu}(\Glie)$, respectively: 
        \begin{gather*}
            (\pi_{\mu} \otimes \mathrm{Id} \otimes \mathrm{Id}) \circ \Delta_{(\mu\epsilon)\nu} = (\pi_{\mu} \otimes \mathrm{Id} \otimes \mathrm{Id}) \circ \Delta_{\mu(\epsilon\nu)}; \\
            (\mathrm{Id} \otimes \mathrm{Id} \otimes \pi_{\nu}) \circ \Delta_{(\mu\epsilon)\nu} = (\mathrm{Id} \otimes \mathrm{Id} \otimes \pi_{\nu}) \circ \Delta_{\mu(\epsilon\nu)}.
        \end{gather*}
\end{theorem}

 As a direct consequence, if $M$ or $N$ is a trivial module, then the identity map is a $Y_{\mu+\epsilon+\nu}(\Glie)$-module isomorphism from $(M\otimes K) \otimes N$ to $M\otimes (K\otimes N)$.

Assume from now on that the middle module $K$ is one-dimensional with structural map $\rho: Y_{\epsilon}(\Glie) \longrightarrow \BC$. The triple tensor product modules $(M\otimes K) \otimes N$ and $M\otimes (K\otimes N)$ are pullbacks of the $Y_{\mu}(\Glie)\otimes Y_{\nu}(\Glie)$-module $M\otimes N$ along the algebra homomorphisms 
$$(\mathrm{Id} \otimes \rho \otimes \mathrm{Id}) \circ \Delta_{(\mu\epsilon)\nu} \ \textrm{ and }\ (\mathrm{Id} \otimes \rho \otimes \mathrm{Id}) \circ \Delta_{\mu(\epsilon\nu)}: Y_{\mu+\epsilon+\nu}(\Glie) \longrightarrow Y_{\mu}(\Glie) \otimes Y_{\nu}(\Glie). $$
We deform the two maps by a formal spectral parameter $z$.
Define the algebra homomorphism $\rho_z: Y_{\epsilon}(\Glie) \longrightarrow \BC[z]$ to be the composition of $\tau_z: Y_{\epsilon}(\Glie) \longrightarrow Y_{\epsilon}(\Glie)[z]$ with $\rho \otimes \mathrm{Id}_{\BC[z]}: Y_{\epsilon}(\Glie)[z] \longrightarrow \BC[z]$. Consider the following two algebra homomorphisms 
$$ (\mathrm{Id} \otimes \rho_z \otimes \mathrm{Id}) \circ \Delta_{(\mu\epsilon)\nu} \ \textrm{ and }\ (\mathrm{Id} \otimes \rho_z \otimes \mathrm{Id}) \circ  \Delta_{\mu(\epsilon\nu)}: Y_{\mu+\epsilon+\nu}(\Glie) \longrightarrow Y_{\mu}(\Glie) \otimes Y_{\nu}(\Glie)[z]. $$
Their evaluations at $z=0$ recover the previous two algebra homomorphisms.

\begin{defi}\cite[Definition 5.4]{Z2} \label{defi: completion}
    Given two $\BQ$-graded vector spaces $M$ and $N$, we define the $z$-completed tensor product $M\otimes_z N$ to be the following vector space
    $$ M \otimes_z N :=  \sum_{\alpha,\beta \in \BQ} \prod_{\gamma \in \BQ_+} (M_{\alpha-\gamma} \otimes N_{\beta+\gamma})((z^{-1})).$$
    Here, for $\alpha \in \BQ$, we let $M_{\alpha}$ denote the $\alpha$-component of $M$.
\end{defi}
We view $M \otimes_z N$ as a completion of the space $M\otimes N[z]$ of polynomials.
If $M$ and $N$ are $\BQ$-graded modules over $\BQ$-graded algebras $\mathscr{A}$ and $\mathscr{B}$ respectively, then $\mathscr{A}\otimes_z \mathscr{B}$ is naturally an algebra that admits a representation on $M\otimes_z N$. 
 
\begin{theorem}\cite[Theorem 5.7, Proposition 5.11]{Z2}   \label{thm: Theta}
    Let $\mu, \epsilon$ and $\nu$ be three coweights and $K$ be a one-dimensional $Y_{\epsilon}(\Glie)$-module with structural map $\rho: Y_{\epsilon}(\Glie) \longrightarrow \BC$. 
    \begin{itemize}
        \item[(1)] In the $z$-completed tensor product $Y_{\mu}(\Glie) \otimes_z Y_{\nu}(\Glie)$ there is a unique element $\Theta_K^{\mu,\nu}(z)$ which is unitriangular of the form 
        \begin{gather*}
            \Theta_K^{\mu,\nu}(z) = \sum_{\beta \in \BQ_+} \Theta_{K,\beta}^{\mu,\nu}(z) \quad \textrm{with} \\
            \Theta_{K,\beta}^{\mu,\nu}(z) \in (Y_{\mu}(\Glie)_{-\beta} \otimes Y_{\nu}(\Glie)_{\beta}) ((z^{-1})) \quad \textrm{ for $\beta \in \BQ_+$ and } \quad \Theta_{K,0}^{\mu,\nu}(z) = 1 \otimes 1
        \end{gather*}
       and which satisfies the intertwining equations for all $x \in Y_{\mu+\epsilon+\nu}(\Glie)$:
      $$ \Theta_K^{\mu,\nu}(z) \times (\mathrm{Id} \otimes \rho_z \otimes \mathrm{Id}) \circ \Delta_{(\mu\epsilon)\nu}(x) = (\mathrm{Id} \otimes \rho_z \otimes \mathrm{Id}) \circ \Delta_{\mu(\epsilon\nu)}(x) \times \Theta_K^{\mu,\nu}(z).  $$
      \item[(2)] For $\beta \in \BQ_+$, the Laurent series $\Theta_{K,\beta}^{\mu,\nu}(z)$ is a polynomial in $z$ with coefficients in $Y_{\mu}^-(\Glie)_{-\beta} \otimes Y_{\nu}^+(\Glie)_{\beta}$ and of degree bounded above by $\langle \epsilon, \beta \rangle$. Furthermore, under the natural identifications $Y_{\mu}^{\pm}(\Glie) = Y_{0}^{\pm}(\Glie)$, we have $\Theta_{K,\beta}^{\mu,\nu}(z) = \Theta_{K,\beta}^{0,0}(z)$.  
    \end{itemize}
\end{theorem}
We call $\Theta_K^{\mu,\nu}(z)$ a Theta series and drop occasionally the superscripts $\mu$ and $\nu$. The polynomiality implies that $\Theta_K(s)$ is a well-defined element of $Y_{\mu}^-(\Glie) \otimes_z Y_{\nu}^+(\Glie)$ for any complex number $s \in \BC$. As a consequence of the intertwining property and the unitriangularity, if $M$ is a $Y_{\mu}(\Glie)$-module and $N$ is a $Y_{\nu}(\Glie)$-module, one of which is finite-dimensional, then the Theta series $\Theta_K(0)$ evaluated at $M \otimes N$ converges to a linear automorphism of $M\otimes N$ which is interpreted as a $Y_{\mu+\epsilon+\nu}(\Glie)$-module isomorphism $$\Theta_K(0)|_{M,N}: (M\otimes K) \otimes N \longrightarrow M \otimes (K\otimes N).$$

\begin{rem} \label{rem: coideal associativity}
    As in Remark \ref{rem: coideal}, we can restrict Theorems \ref{thm: asso trivial} and \ref{thm: Theta} to coideal subalgebras. Let $M$ be a $Y_{\mu}(\Glie)$-module, $K$ be a $Y_{\epsilon}(\Glie)$-module and $N$ be a $Y_{\nu}^{\geq}(\Glie)$-module. Then both triple tensor products $(M\otimes K)\otimes N$ and $M\otimes (K\otimes N)$ are modules over $Y_{\mu+\epsilon+\nu}^{\geq}(\Glie)$ by the coideal subalgebra property.
    \begin{itemize}
        \item[(1)] If $M$ is a trivial module, then the identity map is a $Y_{\mu+\epsilon+\nu}^{\geq}(\Glie)$-module isomorphism of the triple tensor products. 
        \item[(2)] If $K$ is one-dimensional and $M$ or $N$ is finite-dimensional, then the Theta series $\Theta_K(0) \in Y_{\mu}^-(\Glie)\otimes_z Y_{\nu}^+(\Glie)$ evaluated at $M\otimes N$ defines a $Y_{\mu+\epsilon+\nu}^{\geq}(\Glie)$-module isomorphism from $(M\otimes K) \otimes N$ to $M\otimes (K\otimes N)$. 
    \end{itemize}
    Similar statements hold for the coideal subalgebra $Y_{\mu+\epsilon+\nu}^{\leq}(\Glie)$.
\end{rem}

\begin{example}  \label{ex: Theta}
    Fix $i \in I$ and $\mu = \nu = 0$. Let $\Theta_i(z)$ denote the Theta series associated with the algebra homomorphism $Y_{\varpi_i^{\vee}}(\Glie) \longrightarrow \BC$ sending $\xi_j(u)$ to $u^{\delta_{ij}}$; this corresponds to a positive prefundamental representation \cite[Remark 24]{Z1}. As in Example \ref{ex: Theta qaf}, we have the following formula from \cite[Example 5.8]{Z2}:  
    $$ \sum_{n\geq 0} \Theta_{i,n\alpha_j}(z) = \exp\left( \delta_{ij}x_{i,0}^- \otimes x_{i,0}^+\right)\quad \textrm{for $j \in I$}. $$
    Recently, an explicit formula of $\Theta_i(z)$ for $\Glie = \mathfrak{sl}_{r+1}$ has been obtained in \cite{Milot}.
\end{example}

\subsection{S-series}
In \cite{Z2} we constructed another family of generating series, called S-series, in the commutative Drinfeld--Cartan subalgebra of a shifted Yangian. They can be seen as a Yangian analog of the T-series of Frenkel--Hernandez \cite{FH}. In this subsection, we give an intrinsic characterization of S-series.

Given a shifted Yangian $Y_{\mu}(\Glie)$, consider the normalized Drinfeld--Cartan series $\overline{\xi}_i(u)$, for $i \in I$. They are power series in $u^{-1}$ of constant term 1 with coefficients in the commutative Drinfeld--Cartan subalgebra $Y_{\mu}^0(\Glie)$. Let $H_{\mu}$ denote the subspace of $Y_{\mu}^0(\Glie)$ linearly spanned by the coefficients of the logarithms $\log(\overline{\xi}_i(u))$ for all $i \in I$. All the shift homomorphisms of \eqref{def: shift homo} preserve $\overline{\xi}_i(u)$ and hence its logarithm. They induce a natural identification of vector spaces
$$ H_{\mu} \cong H_0,\qquad  \log \overline{\xi}_i(u) \mapsto \log \xi_i(u). $$
\begin{prop}  \label{prop: S-series}
For $\mu$ a coweight and $i \in I$,
there exists a unique power series 
$$\mathcal{S}_i^{\mu}(z) \in \exp\left( z^{-1} H_{\mu}[[z^{-1}]] \right) \subset Y_{\mu}^0(\Glie)[[z^{-1}]]$$
satisfying the following commutation relations with the Drinfeld generators $x_{j,n}^{\pm} \in Y_{\mu}(\Glie)$:
   \begin{gather*}
       \mathcal{S}_i^{\mu}(z) x_{j,n}^{\pm} = x_{j,n}^{\pm} \mathcal{S}_i^{\mu}(z) \quad \textrm{if $j \neq i$,}  \\
       \mathcal{S}_i^{\mu}(z) x_{i,n}^- = (x_{i,n}^- + z^{-1}x_{i,n+1}^-) \mathcal{S}_i^{\mu}(z), \quad x_{i,n}^+ \mathcal{S}_i^{\mu}(z) = \mathcal{S}_i^{\mu}(z) (x_{i,n}^+ + z^{-1}x_{i,n+1}^+).
   \end{gather*}
 Under the natural identification $H_{\mu} = H_0$, we have $\mathcal{S}_i^{\mu}(z) = \mathcal{S}_i^0(z)$.
\end{prop}
\begin{proof}
    Our $\mathcal{S}_i^{\mu}(z)$ corresponds to $S_i(-z)$ in \cite[Proposition 4.1]{Z2}; see Equation (4.10), the proof of Lemma 5.1 and Step 4 of the proof of Theorem 5.7 therein. For uniqueness, it suffices to show that any element of $H_{\mu}$ that commutes with all the $x_{j,n}^{\pm}$ must be zero. This follows from \cite[Eq.(2.3)]{Z2}.
\end{proof}
We add a minus sign to the S-series of \cite{Z2} to simplify the formulas of R-matrices in the next subsection. The spectral parameter $z$ is used instead of the formal variable $u$ of generating series. This is to emphasize that S-series are essentially R-matrices. We comment that evaluations of S-series at highest weight representations of shifted Yangians were studied earlier in \cite{HZ} and \cite{GW} and were characterized as solutions to additive difference equations defined by another family of modified Drinfeld--Cartan series, due to Gerasimov--Kharchev--Lebedev--Oblezin \cite{GKLO}. 

We attach S-series to one-dimensional modules of shifted Yangians.   
Recall first from \cite[\S 4.1]{HZ} the bijection between $I$-tuples of complex monic polynomials in one variable and isomorphism classes of one-dimensional modules over shifted Yangians 
$$ \mathbf{p} \leadsto L(\mathbf{p}). $$
Here $\mathbf{p} = (\mathbf{p}_i(u))_{i\in I}$ and $\mathbf{p}_i(u)$ is a monic polynomial in $\BC[u]$ for $i \in I$. The corresponding one-dimensional module $L(\mathbf{p})$ is defined over the dominantly shifted Yangian $Y_{\epsilon}(\Glie)$ with $\epsilon := \sum_{i\in I} \deg (\mathbf{p}_i) \varpi_i^{\vee}$ by the algebra homomorphism $Y_{\epsilon}(\Glie) \longrightarrow \BC$ sending $\xi_i(u)$ to $\mathbf{p}_i(u)$ and $x_i^{\pm}(u)$ to zero. We call $\mathbf{p}$ the $\ell$-weight of the one-dimensional module. 
 
Given a one-dimensional $Y_{\epsilon}(\Glie)$-module $K$ of $\ell$-weight $(\mathbf{p}_i(u))_{i\in I}$, let us first factorize the monic polynomials in the $\ell$-weights
$$ \mathbf{p}_i(u) = (u-a_{i1}) (u-a_{i2}) \cdots (u-a_{i,k_i}) \quad \textrm{for $i \in I$}. $$ 
For $\mu$ a coweight, the S-series associated to the module $K$ is defined as \cite[(4.12)]{Z2}
$$ \mathcal{S}_K^{\mu}(z) := \prod_{i\in I} \mathcal{S}_i^{\mu}(z - a_{i1}) \mathcal{S}_i^{\mu}(z-a_{i2}) \cdots \mathcal{S}_i^{\mu}(z-a_{i,k_i}) \in  1 + z^{-1} Y_{\mu}^0(\Glie)[[z^{-1}]]. $$

\begin{rem}  \label{rem: Theta Omega}
    Theta series can be computed from the shifted coproduct of S-series. Given two coweights $\mu$ and $\nu$, we take the factorization \cite[Definition 5.2]{Z2} 
$$ \Delta_{\mu,\nu}(\mathcal{S}_K^{\mu+\nu}(z)) = (1\otimes \mathcal{S}_K^{\nu}(z)) \times \mathbf{\Omega}_K^{\mu,\nu}(z) \times (\mathcal{S}_K^{\mu}(z) \otimes 1). $$
The middle factor $\mathbf{\Omega}_K^{\mu,\nu}(z)$ at the right-hand side is a power series in $z^{-1}$ with coefficients in $Y_{\mu}(\Glie)\otimes Y_{\nu}(\Glie)$ and of constant term $1\otimes 1$. It is also unitriangular as Theta series. Restricted to weight components, we have
$$ \Theta_{K,\beta}^{\mu,\nu}(z) = \mathbf{\Omega}_{K,\beta}^{\mu,\nu}(-z) \times \prod_{i\in I} \mathbf{p}_i(-z)^{\langle \varpi_i^{\vee},\beta\rangle} \quad \textrm{for $\beta \in \BQ_+$}. $$
As an example, let $a \in \BC$ and $K$ be the one-dimensional $Y_{2\varpi_1^{\vee}}(\mathfrak{sl}_2)$-module of $\ell$-weight $u(u-a)$. Then from Example \ref{ex: Theta}(2) we know the coproduct of $\mathcal{S}_1(z)$ and hence of $\mathcal{S}_K(z) = \mathcal{S}_1(z)\mathcal{S}_1(z-a)$. This in turn gives us the Theta series
    $$ \Theta_K(z) = \exp\left(x_{1,0}^-\otimes (x_{1,1}^+-zx_{1,0}^+) \right) \times \exp\left((x_{1,1}^--(z+a)x_{1,0}^-) \otimes x_{1,0}^+ \right). $$
\end{rem}

\subsection{R-matrices for one-dimensional modules}  \label{ss: one-dim R}
Fix $K$ to be a one-dimensional module over a dominantly shifted Yangian $Y_{\epsilon}(\Glie)$. For $N$ a graded module over a shifted Yangian $Y_{\nu}(\Glie)$, we recall from \cite[\S 4]{Z2} the construction of $Y_{\epsilon+\nu}(\Glie)$-module morphisms between completions of tensor product modules $K\otimes N$ and $N\otimes K$.

Recall that the shifted Yangian $Y_{\nu}(\Glie)$ is $\BQ$-graded. By a graded module over this algebra, we mean an ordinary module $N$ that is a direct sum of vector subspaces $N_{\beta}$, for $\beta \in \BQ$, such that 
$$ Y_{\nu}(\Glie)_{\gamma} N_{\beta} \subset N_{\gamma+\beta} \quad \textrm{for $\beta, \gamma \in \BQ$.} $$
In this situation, consider the subspace of Laurent series in $z^{-1}$  with coefficients in $N$:
$$ N^z := \sum_{\beta \in \BQ} N_{\beta}((z^{-1})) \subset N((z^{-1})). $$
Alternatively, $N^z$ is the $z$-completed tensor product $\BC \otimes_z N$ where the $\BQ$-graded one-dimensional vector space $\BC$ is concentrated in degree zero. It contains $N[z]$.

Let $\rho: Y_{\epsilon}(\Glie) \longrightarrow \BC$ be the structural map of $K$. The space $N[z]$ of polynomials is naturally a $Y_{\nu}(\Glie)[z]$-module by scalar extension. Its pullbacks along the following two algebra homomorphisms
$$ (\rho \otimes \tau_z) \circ \Delta_{\epsilon,\nu} \quad \textrm{and}\quad (\tau_z \otimes \rho) \circ \Delta_{\nu,\epsilon}: Y_{\epsilon+\nu}(\Glie) \longrightarrow Y_{\nu}(\Glie)[z] $$
define two $Y_{\epsilon+\mu}(\Glie)$-modules, called deformed tensor products, and denoted by $K\otimes N_z$ and $N_z \otimes K$, respectively. Their evaluations at $z = 0$ recover the usual tensor product modules $K \otimes N$ and $N\otimes K$. By replacing $N[z]$ with the larger space $N^z$, we obtain two $Y_{\epsilon+\nu}(\Glie)$-modules, denoted by $\overline{K \otimes N_z}$ and $\overline{N_z \otimes K}$, respectively. They contain the deformed tensor product modules $K \otimes N_z$ and $N_z \otimes K$ as submodules.

Let $(\Bp_i(u))_{i\in I}$ be the $\ell$-weight of the one-dimensional module $K$.
Define the linear automorphism $\mathcal{D}_{K,N}(z)$ of $N^z$ by
$$ \mathcal{D}_{K,N}(z) (v) = v \times \prod_{i\in I} \mathbf{p}_i(z)^{-\langle \varpi_i^{\vee},\beta\rangle} \quad \textrm{for $\beta \in \BQ$ and $v \in N_{\beta}((z^{-1}))$.} $$
We also evaluate the power series $\mathcal{S}_K^{\nu}(z) \in Y_{\nu}^0(\Glie)[[z^{-1}]]$ at $N^z$ to obtain another linear automorphism $\mathcal{S}_K(z)|_N$ and then take their composition \begin{equation*} 
    \CR_{K,N}(z) := \mathcal{D}_{K,N}(z) \circ \mathcal{S}_K^{\nu}(z)|_N \in \mathrm{Aut}(N^z). 
\end{equation*}
 
\begin{theorem} \cite[Theorem 4.6]{Z2} \label{thm: one-dim R}
    For $K$ a one-dimensional $Y_{\epsilon}(\Glie)$-module and $N$ a root graded $Y_{\nu}(\Glie)$-module, the linear automorphism $\CR_{K,N}(z)$ of $N^z$ is a $Y_{\epsilon+\nu}(\Glie)$-module isomorphism of the $z$-completed tensor product modules
$$ \CR_{K,N}(z): \overline{K\otimes N_z} \longrightarrow \overline{N_z \otimes K}. $$ 
\end{theorem}
To be precise, $\CR_{K,N}(z)$ was proved in \cite[Theorem 4.6]{Z2} to be a module isomorphism from completed $K_{-z} \otimes N$ to completed $N\otimes K_{-z}$ on the same underlying space $N^z$. This is equivalent to the above theorem by \cite[Proposition 6.2]{Z2}, which is a special case of Proposition \ref{prop: coproduct spectral} below. 

We are interested in the case where $N$ is the regular representation of $Y_{\nu}(\Glie)$ whose root grading is the weight grading. In this case, $Y_{\nu}(\Glie)^z$ is the sum of $Y_{\nu}(\Glie)_{\beta}((z^{-1}))$ over $\beta \in \BQ$. It is a subalgebra of $Y_{\nu}(\Glie)((z^{-1}))$ that contains $Y_{\nu}(\Glie)[z]$. As a consequence of $\CR_{K,Y_{\nu}(\Glie)}(z)$ being a module isomorphism, we have the following equations of linear endomorphisms of $Y_{\nu}(\Glie)^z$; see also \cite[(4.14) \& comments above Definition 4.3]{Z2}:
\begin{equation}  \label{equ: conjugation T Yangian} 
  \begin{split}
     \CR_{K,Y_{\nu}(\Glie)}(z) \circ \xi_{i,s} \circ \CR_{K,Y_{\nu}(\Glie)}(z)^{-1} = \xi_{i,s}, \\
    \CR_{K,Y_{\nu}(\Glie)}(z) \circ x_{i,n}^- \circ \CR_{K,Y_{\nu}(\Glie)}(z)^{-1}  = \mathbf{p}_i(z+\sigma_i^-) (x_{i,n}^-),  \\
     \CR_{K,Y_{\nu}(\Glie)}(z) \circ x_{i,n}^+ \circ \CR_{K,Y_{\nu}(\Glie)}(z)^{-1} = \frac{1}{\mathbf{p}_i(z+\sigma_i^+)}(x_{i,n}^+).
    \end{split}
\end{equation}
Here, an element $x$ of the algebra $Y_{\nu}(\Glie)^z$  is identified with the linear endomorphism of $Y_{\nu}(\Glie)^z$ defined by left multiplication $y \mapsto x y$. Let us explain that the three terms at the right-hand side belong to $Y_{\nu}(\Glie)^z$. The first term is clear. As in \cite[\S 2.6]{GTL0}, define $\sigma_i^{\pm}$ to be algebra endomorphisms of $Y_{\nu}^{\pm}(\Glie)$ sending $x_{j,n}^{\pm}$ to $x_{j,n+\delta_{ij}}^{\pm}$. View $\mathbf{p}_i(z+\sigma_i^-)$ as an element of $\BC[\sigma_i^-][z]$ that maps $Y_{\nu}^-(\Glie)$ to $Y_{\nu}^-(\Glie)[z]$, then the second term is in $Y_{\nu}(\Glie)[z]$. Similarly, view $\frac{1}{\mathbf{p}_i(z+\sigma_i^+)}$ as an element of $\BC[\sigma_i^+][[z^{-1}]]$, then the third term is in $Y_{\nu}(\Glie)_{\alpha_i}[[z^{-1}]]$. As a consequence, the conjugation map $x \mapsto \CR_{K,Y_{\mu}(\Glie)}(z)^{-1} \circ x \circ \CR_{K,Y_{\mu}(\Glie)}(z)$ restricts to a homomorphism of $\BC[z]$-algebras
\begin{equation}  \label{equ: psi-}
    \mathrm{C}_{K,z}^-:  Y_{\mu}^-(\Glie)[z] \longrightarrow Y_{\mu}^-(\Glie)((z^{-1})), \quad x_{i,n}^- \mapsto \frac{1}{\mathbf{p}_i(z+\sigma_i^-)}(x_{i,n}^-).
\end{equation}
Similarly, the conjugation map $x \mapsto \CR_{K,Y_{\mu}(\Glie)}(z) \circ x \circ \CR_{K,Y_{\mu}(\Glie)}(z)^{-1}$ restricts to a homomorphism of $\BC[z]$-algebras
\begin{equation}  \label{equ: psi+}
    \mathrm{C}_{K,z}^+:  Y_{\mu}^+(\Glie)[z] \longrightarrow Y_{\mu}^+(\Glie)((z^{-1})), \quad x_{i,n}^+ \mapsto \frac{1}{\mathbf{p}_i(z+\sigma_i^+)}(x_{i,n}^+). 
\end{equation}

\section{Compatibility with spectral parameter automorphisms}  \label{sec: spectral parameter coproduct}
It is well-known that the spectral parameter automorphisms $\tau_a$, for $a\in \BC$, are Hopf algebra automorphisms of the ordinary Yangian. In this section, we prove a similar statement for shifted Yangians. 

\begin{prop}  \label{prop: coproduct spectral}
    For $\mu$ and $\nu$ two coweights and $a \in \BC$, we have 
    $$ (\tau_a \otimes \tau_a) \circ \Delta_{\mu,\nu} = \Delta_{\mu,\nu} \circ \tau_a : Y_{\mu+\nu}(\Glie) \longrightarrow Y_{\mu}(\Glie) \otimes Y_{\nu}(\Glie).$$
\end{prop}
\begin{proof}
Fix $a\in \BC$ and define the algebra homomorphism $\Delta_{\mu,\nu}' := (\tau_a \otimes \tau_a)\circ\Delta_{\mu,\nu} \circ \tau_a^{-1}$ from $Y_{\mu+\nu}(\Glie)$ to $Y_{\mu}(\Glie) \otimes Y_{\nu}(\Glie)$. It suffices to show that the family $(\Delta_{\mu,\nu}')_{\mu,\nu\in \BP^{\vee}}$ satisfies all the properties of Theorem \ref{thm: coproduct shifted Yangian}. 

    Part (4) is almost trivial. Part (1) follows from the well-known fact that $\tau_a$ is a Hopf algebra automorphism of the ordinary Yangian $Y_0(\Glie)$. 

    For Part (2), assume $0\leq m < -\langle \nu, \alpha_i\rangle$. Since $\tau_a^{-1}(x_{i,m}^+)$ is a linear combination of $x_{i,k}^+$ for $0\leq k \leq m$, we have 
    $$ \Delta_{\mu,\nu} (\tau_a^{-1}(x_{i,m}^+)) = \tau_a^{-1}(x_{i,m}^+) \otimes 1, \quad \Delta_{\mu,\nu}'(x_{i,m}^+) = x_{i,m}^+ \otimes 1. $$
    The case of $x_{i,n}^-$ is similar.

    For Part (3), we deform the shift homomorphism $\iota_{\epsilon,\eta}^{\mu}$ of Eq.\eqref{def: shift homo}, for $\epsilon$ and $\eta$ anti-dominant, to another algebra homomorphism $\jmath_{\epsilon,\eta}^{\mu} := \tau_a^{-1} \circ \iota_{\epsilon,\eta}^{\mu} \circ \tau_a$ from $Y_{\mu}(\Glie)$ to $Y_{\mu+\epsilon+\eta}(\Glie)$. Then we are reduced to an equation of algebra homomorphisms
    $$ (\sharp)_{\mu,\nu}:\qquad (\jmath_{\epsilon,0}^{\mu} \otimes \jmath_{0,\eta}^{\nu}) \circ \Delta_{\mu,\nu} = \Delta_{\mu+\epsilon,\nu+\eta} \circ \jmath_{\epsilon,\eta}^{\mu+\nu}.  $$

\noindent {\it Claim 1.} For $\epsilon,\epsilon',\eta,\eta'$ anti-dominant coweights, we have 
$$ \jmath_{\epsilon',\eta'}^{\mu+\epsilon+\eta} \circ  \iota^{\mu}_{\epsilon,\eta} = \iota_{\epsilon,\eta}^{\mu+\epsilon'+\eta'} \circ \jmath_{\epsilon',\eta'}^{\mu}: Y_{\mu}(\Glie) \longrightarrow Y_{\mu+\epsilon+\epsilon'+\eta+\eta'}(\Glie). $$

 \begin{proof}
    We check the images of the generating series $x_i^+(u)$, for $i \in I$, of $Y_{\mu}(\Glie)$ under the two algebra homomorphisms at both sides; the cases $x_i^-(u)$ and $\xi_i(u)$ are similar. Set $s = -\langle \epsilon, \alpha_i\rangle$ and $t = - \langle \epsilon',\alpha_i\rangle$. Then $s$ and $t$ are non-negative integers and
    \begin{align*}
        \textrm{left-hand side: } \quad & x_i^+(u)  \mapsto \langle u^s x_i^+(u)\rangle_+ \mapsto \langle u^s \langle (u-a)^t x_i^+(u) \rangle_+   \rangle_+; \\
        \textrm{right-hand side: } \quad & x_i^+(u) \mapsto \langle (u-a)^t x_i^+(z) \rangle_+ \mapsto \langle (u-a)^t \langle z^s x_i^+(u) \rangle_+ \rangle_+.
    \end{align*}
    Here we used the fact that $\langle f\rangle_+^u = \langle f\rangle_+^{u-a}$ for $f$ a vector-valued formal Laurent series in $u$, viewed also as a formal Laurent series in $u-a$.  
    Conclude by the identity $\langle P\langle Q f\rangle_+\rangle_+ = \langle PQ f\rangle_+$ for $P, Q \in \BC[u]$.
 \end{proof}   

\noindent {\it Claim 2.} For $\epsilon$ and $\eta$ anti-dominant coweights and $b \in \BC$, we have
$$ (\tau_b \otimes \tau_b) \circ \Delta_{\epsilon,\eta} = \Delta_{\epsilon,\eta} \circ \tau_b: Y_{\epsilon+\eta}(\Glie) \longrightarrow Y_{\epsilon}(\Glie) \otimes Y_{\eta}(\Glie). $$
\begin{proof}
   For $i \in I$, set $s_i = -\langle \epsilon, \alpha_i\rangle$ and $t_i = - \langle \eta, \alpha_i\rangle$; both are non-negative integers.
    The shifted Yangian $Y_{\epsilon+\eta}(\Glie)$ as an algebra is generated by $x_{i,0}^{\pm}$ and $\xi_{i,s_i+t_i+1}$. We compute the images of these generators under the two algebra homomorphisms based on the explicit coproduct formulas in the anti-dominant case \cite[\S 4]{coproduct}. 
    
    For $x_{i,0}^+$, the two images are equal to
     $x_{i,0}^+ \otimes 1 + \delta_{0,s_i} \otimes x_{i,0}^+$. Similarly, for $x_{i,0}^-$, both images are equal to $1\otimes x_{i,0}^- + x_{i,0}^- \otimes \delta_{0,t_i}$.   
     
     For $\xi_{i,s_i+t_i}$, under the identifications $Y_{\mu}^{\pm}(\Glie) = Y_0^{\pm}(\Glie)$, the left-hand side gives  
     \begin{align*}
         \xi_{i,s_i+t_i+1} &\mapsto 1\otimes \xi_{i,t_i+1} + \xi_{i,s_i+1} \otimes 1 + \xi_{i,s_i} \otimes \xi_{i,t_i} - \sum_{\beta \in \Phi} (\alpha_i,\beta) x_{\beta}^- \otimes x_{\beta}^+   \\
         &\mapsto 1 \otimes (\xi_{i,t_i+1} + (t_i+1) b \xi_{i,t_i} + \frac{t_i(t_i+1)}{2} b^2 )  \\
         &\qquad + (\xi_{i,s_i+1} + (s_i+1)b \xi_{i,s_i} + \frac{s_i(s_i+1)}{2}b^2) \otimes 1 \\
         & \qquad + (\xi_{i,s_i}+s_i b) \otimes (\xi_{i,t_i} + t_i b) - \sum_{\beta\in \Phi} (\alpha_i,\beta) x_{\beta}^- \otimes x_{\beta}^+.
     \end{align*} 
     The right-hand side gives
     \begin{align*}
         \xi_{i,s_i+t_i+1} &\mapsto \xi_{i,s_i+t_i+1} + (s_i+t_i+1)b \xi_{i,s_i+t_i} + \frac{(s_i+t_i)(s_i+t_i+1)}{2}b^2  \\
         &\mapsto 1\otimes \xi_{i,t_i+1} + \xi_{i,s_i+1} \otimes 1 + \xi_{i,s_i} \otimes \xi_{i,t_i} - \sum_{\beta \in \Phi} (\alpha_i,\beta) x_{\beta}^- \otimes x_{\beta}^+ \\
         &\qquad + (s_i+t_i+1) b (\xi_{i,s_i} \otimes 1 + 1 \otimes \xi_{i,t_i}) + \frac{(s_i+t_i)(s_i+t_i+1)}{2}b^2 \otimes 1.
     \end{align*}
    Both sides coincide. 
\end{proof}
    
    Let $\epsilon'$ and $\eta'$ be anti-dominant coweights such that both $\mu+\epsilon'$ and $\nu+\eta'$ are anti-dominant. Applying the injective map $\iota_{\epsilon',0}^{\mu} \otimes \iota_{0,\eta'}^{\nu}$ to $(\sharp)_{\mu,\nu}$ and making use of Claim 1      we see that $(\sharp)_{\mu,\nu}$ is equivalent to the following equation 
    $$ (\sharp)_{\mu+\epsilon',\nu+\eta'}:\qquad (\jmath_{\epsilon,0}^{\mu+\epsilon'} \otimes \jmath_{0,\eta}^{\nu+\eta'}) \circ \Delta_{\mu+\epsilon',\nu+\eta'} = \Delta_{\mu+\epsilon'+\epsilon,\nu+\eta'+\eta} \circ \jmath_{\epsilon,\eta}^{\mu+\epsilon+\nu+\eta}. $$
    Rewriting the deformed shift homomorphisms $\jmath$ in terms of the shift homomorphisms $\iota$ and using Claim 2 to move $\tau_a \otimes \tau_a$ to the left and $\tau_a^{-1} = \tau_{-a}$ at the right, we are reduced to the known equation of Theorem \ref{thm: coproduct shifted Yangian}(3).
\end{proof}

The spectral parameter automorphism $\tau_a$ of a non-zero shifted $Y_{\mu}(\Glie)$ does not send an S-series to another S-series. Nevertheless, it is compatible with Theta series.

\begin{cor}
    The Theta series of Theorem \ref{thm: Theta} satisfy the following compatibility with spectral parameter automorphisms: 
    $$ (\tau_a\otimes \tau_a)(\Theta_{K,\beta}^{\mu,\nu}(z)) = \Theta_{K,\beta}^{\mu,\nu}(z-a) \in Y_{\mu}^-(\Glie)_{-\beta} \otimes Y_{\nu}^+(\Glie)_{\beta}[z] \quad \textrm{for $a \in \BC$}. $$
    As a consequence, $\Theta_K(z)$ is obtained from $\Theta_K(0)$ by applying $\tau_{-z}\otimes \tau_{-z}$.
\end{cor}
\begin{proof}
    It suffices to check that the left-hand side solves the intertwining property of $\Theta_K(z-a)$ in Theorem \ref{thm: Theta}(1). By Proposition \ref{prop: coproduct spectral}, we have:
    $$ (\tau_a \otimes \tau_a) \circ (\mathrm{Id} \otimes \rho_z \otimes \mathrm{Id}) \circ \Delta_{(\mu\epsilon)\nu} = (\mathrm{Id}\otimes \rho_{z-a} \otimes \mathrm{Id}) \circ \Delta_{(\mu\epsilon)\nu} \circ \tau_a.  $$
    Similar relation holds for $\Delta_{\mu(\epsilon\nu)}$. By applying $\tau_a \otimes \tau_a$ to the intertwining property of $\Theta_K(z)$, we get the desired intertwining property of $\Theta_K(z-a)$.
\end{proof}

\section{The unitriangular R-matrix for shifted Yangians}  \label{sec: uni R shifted Yangian}
The ordinary Yangian possesses a universal R-matrix $\CR(z)$ that admits a Gauss decomposition, due to Gautam--Toledano Laredo--Wendlandt \cite{GTLW}, into a product of three power series of constant term $1\otimes 1$:
$$ \CR(z) = \CR^+(z) \times \CR^0(z) \times \CR^-(z) \in Y(\Glie)^{\otimes 2}[[z^{-1}]].  $$
The abelian part $\CR^0(z)$ has coefficients in the commutative subalgebra $Y^0(\Glie) \otimes Y^0(\Glie)$, one of the unitriangular parts $\CR^-(z)$ has coefficients in $Y^-(\Glie) \otimes Y^+(\Glie)$, and the other unitriangular part $\CR^+(z)$ is simply $\CR^-(-z)^{-1}_{21}$ where $21$ denotes the flip of the two tensor factors. For this reason, we are concentrated on $\CR^-(z)$, and call it {\it the} unitriangular R-matrix of the Yangian.

In this section, building on the construction of \cite{GTLW}, we extend the unitriangular R-matrix $\CR^-(z)$ to shifted Yangians and prove a uniqueness statement for the evaluation of $\CR^-(z)$ at the tensor product of the regular representation of a shifted Yangian with an arbitrary representation of another shifted Yangian.

\begin{theorem}\cite[Theorem 4.1]{GTLW}   \label{thm: R- GTLW}
    There exists a unique power series in $z^{-1}$ with coefficients in $Y^-(\Glie) \otimes Y^+(\Glie)$ of constant term $1\otimes 1$, which is unitriangular:
    \begin{gather*}
        \CR^-(z) = \sum_{\beta \in \BQ_+} \CR_{\beta}^-(z) \quad \textrm{such that} \\
        \CR_{\beta}^-(z) \in (Y^-(\Glie)_{-\beta}\otimes Y^+(\Glie)_{\beta})[[z^{-1}]] \quad \textrm{for}\quad \beta \in \BQ_+ \quad \textrm{and}\quad \CR_0^-(z) = 1\otimes 1
    \end{gather*}
    and which satisfies the intertwining equations in $Y(\Glie)^{\otimes 2}((z^{-1}))[[u^{-1}]]$ for $i \in I$:
    \begin{gather*}
    \CR^-(z) \times (\tau_z \otimes \mathrm{Id})\circ \Delta(\xi_i(u)) = (\xi_i(u-z)\otimes \xi_i(u)) \times \CR^-(z), \\
    \CR^-(-z) \times (\mathrm{Id} \otimes \tau_z)\circ \Delta(\xi_i(u)) = (\xi_i(u)\otimes \xi_i(u-z)) \times \CR^-(-z).
    \end{gather*}
\end{theorem}
Let $\mu$ and $\nu$ be two coweights. Under the natural identifications of the subalgebras $Y_{\mu}^{\pm}(\Glie)$ of the shifted Yangian with the subalgebras $Y^{\pm}(\Glie)$ of the ordinary Yangian, we view the power series $\CR^-(z)$ of the above theorem in
$$ \CR^-(z) \in (Y_{\mu}^-(\Glie) \otimes Y_{\nu}^+(\Glie))[[z^{-1}]] $$
with the same unitriangularity property. 
\begin{prop}  \label{prop: R- intertwining}
    Given two coweights $\mu$ and $\nu$, we have the following intertwining equations in $(Y_{\mu}^{\leq}(\Glie)\otimes Y_{\nu}^{\geq}(\Glie))((z^{-1}))((u^{-1}))$ for $i \in I$:
    \begin{gather*}
        \CR^-(z) \times (\tau_z \otimes \mathrm{Id})\circ \Delta_{\mu,\nu}(\xi_i(u)) = (\xi_i(u-z)\otimes \xi_i(u)) \times \CR^-(z), \\
        \CR^-(-z) \times (\mathrm{Id} \otimes \tau_z)\circ \Delta_{\mu,\nu}(\xi_i(u)) = (\xi_i(u) \otimes \xi_i(u-z)) \times \CR^-(-z).
    \end{gather*}
\end{prop}
\begin{proof}
We shall prove the first intertwining equation; the same idea works for the second equation. By Lemma \ref{lem: Yangian coproduct estimation} and Theorem \ref{thm: R- GTLW}, all the Laurent series have coefficients in the subalgebra $Y_{\mu}^{\leq}(\Glie)\otimes Y_{\nu}^{\geq}(\Glie)$.
    Given a complex number $a$ and an anti-dominant coweight $\epsilon$, consider the algebra homomorphism in the proof of Proposition \ref{prop: coproduct spectral}: $$ \tau_a \circ \iota_{\epsilon,0}^{\mu} \circ \tau_{-a}: Y_{\mu}(\Glie) \longrightarrow Y_{\mu+\epsilon}(\Glie) $$
    It acts on the generating series of $Y_{\mu}(\Glie)$ as follows:
    $$ x_i^-(u) \mapsto x_i^-(u),\quad x_i^+(u) \mapsto \langle (u+a)^{-\langle \epsilon,\alpha_i\rangle} x_i^+(u)\rangle_+,\quad \xi_i(u) \mapsto (u+a)^{-\langle \epsilon,\alpha_i\rangle} \xi_i(u). $$
    In particular, it is the evaluation at $z = a$ of the following algebra homomorphism
    $$ \iota_{\epsilon,0}^{\mu}(z): Y_{\mu}(\Glie) \longrightarrow Y_{\mu+\epsilon}(\Glie)[z] $$
    which is injective by the injectivity of $\iota_{\epsilon,0}^{\mu}$. 

    For $s \leq 0$, let $\CR_s^- \in Y_{\mu}^-(\Glie) \otimes Y_{\nu}^+(\Glie)$ denote the coefficient of $z^s$ in the power series $\CR^-(z)$; it is independent of $\mu$ and $\nu$. Fix from now on $i \in I$. Let us develop the second and third factors of the $i$th intertwining equation as Laurent series in $z$ and $u^{-1}$:
    $$ (\tau_z \otimes \mathrm{Id})\circ \Delta_{\mu,\nu}(\xi_i(u)) = \sum_{k} \sum_{n\geq 0} f_{k,n}^{\mu,\nu} z^n u^{-k-1},\quad \xi_i(u-z) \otimes \xi_i(u) = \sum_k \sum_{n\geq 0} g_{k,n}^{\mu,\nu}z^n u^{-k-1}. $$
    Then $k$ is bounded from below, and for each $k$, the coefficients $f_{k,n}^{\mu,\nu}$ and $g_{k,n}^{\mu,\nu}$ are nonzero for finitely many $n$. The $i$th intertwining equation is rewritten as
    $$   \sum_k u^{-k-1} \sum_{t} z^t \left(\sum_{s+n=t} \CR_s f_{k,n}^{\mu,\nu} \right) = \sum_k u^{-k-1} \sum_{t} z^t \left(\sum_{s+n=t} g_{k,n}^{\mu,\nu} \CR_s \right). $$
    For fixed $k$, the second summation over $t$ is bounded from above, the third summation is finite and belongs to the tensor product algebra $Y_{\mu}(\Glie) \otimes Y_{\nu}(\Glie)$. So the intertwining equation becomes the system of equations for $k, t \in \BZ$
    $$ (\sharp)^{\mu,\nu}_{k,t}: \quad \sum_{s+n=t} \CR_s \ f_{k,n}^{\mu,\nu} = \sum_{s+n=t} g_{k,n}^{\mu,\nu} \CR_s \in Y_{\mu}(\Glie)\otimes Y_{\nu}(\Glie). $$

    Given two anti-dominant coweights $\epsilon$ and $\eta$, we apply the injective algebra homomorphism $\iota_{\epsilon,0}^{\mu}(z) \otimes \iota_{0,\eta}^{\nu}: Y_{\mu}(\Glie) \otimes Y_{\nu}(\Glie) \longrightarrow Y_{\mu+\epsilon}(\Glie)\otimes Y_{\nu+\eta}(\Glie)[z]$ to this system. Let us determine the resulting equivalent system. Notice first that $\CR^-_s \in Y^-(\Glie) \otimes Y^+(\Glie)$ is fixed by this homomorphism. Next, to compute the images of the $f_{k,n}^{\mu,\nu}$ by this homomorphism, let us evaluate $z$ at an arbitrary complex number $a$:
    \begin{align*}
       &\quad (\tau_a \circ \iota_{\epsilon,0}^{\mu}\circ \tau_{-a} \otimes \iota_{0,\eta}^{\nu}) \circ (\tau_a \otimes \mathrm{Id}) \circ \Delta_{\mu,\nu}(\xi_i(u))  \\
       &= (\tau_a \otimes \mathrm{Id}) \circ (\iota_{\epsilon,0}^{\mu}\otimes \iota_{0,\eta}^{\nu}) \circ \Delta_{\mu,\nu}(\xi_i(u)) = (\tau_a \otimes \mathrm{Id}) \circ \Delta_{\mu+\epsilon,\nu+\eta} \circ \iota_{\epsilon,\eta}^{\mu+\nu}(\xi_i(u)) \\
       &= (\tau_a\otimes \mathrm{Id}) \circ \Delta_{\mu+\epsilon,\nu+\eta}(u^{-\langle \epsilon+\eta,\alpha_i\rangle} \xi_i(u)) = \sum_k \sum_{n\geq 0} f_{k-\langle \epsilon+\eta,\alpha_i\rangle,n}^{\mu+\epsilon,\nu+\eta} a^n u^{-k-1}.
    \end{align*}
   By polynomiality in $z$, we see that the image of $f_{k,n}^{\mu,\nu}$ under the map $\iota_{\epsilon,0}^{\mu}(z) \otimes \iota_{0,\eta}^{\eta}$ is the constant polynomial $f_{k-\langle \epsilon+\eta,\alpha_i\rangle,n}^{\mu+\epsilon,\nu+\eta}$. Similar computations can be done for $g_{k,n}^{\mu,\nu}$. As a consequence, we get an equivalence
   $$(\sharp)_{k,t}^{\mu,\nu}  \quad \Longleftrightarrow \quad  (\sharp)^{\mu+\epsilon,\nu+\eta}_{k-\langle \epsilon+\eta,\alpha_i\rangle, t} \quad \textrm{for $k,\, t \in \BZ$ and $\epsilon,\, \eta$ antidominant}. $$
   Since $(\sharp)_{k,t}^{0,0}$ holds true for all $k, t \in \BZ$ by Theorem \ref{thm: R- GTLW}, standard zigzag arguments as in the proof of \cite[Theorem 4.12]{coproduct} show that $(\sharp)_{k,t}^{\mu,\nu}$ holds true for all $k, t \in \BZ$ and all coweights $\mu, \nu \in \BP^{\vee}$.
\end{proof}
The idea of the proof is to apply the injective algebra homomorphism $\iota_{\epsilon,0}^{\mu}(z) \otimes \iota_{0,\eta}^{\nu}$ to the intertwining equations to perform zigzag arguments. This can be done for the second and third factors because they are polynomial in $z$, but not for $\CR^-(z)$ which is a power series in $z^{-1}$, so extra care is needed. One may deduce the second intertwining equation from the first equation by substituting $z\mapsto -z$, applying $\tau_{z}\otimes \tau_{z}$ and using the translation invariance \cite[Theorem 4.1(2)]{GTLW}. Again, a similar issue is that $\tau_z\otimes \tau_z$ cannot be applied directly to the power series $\CR^-(z)$. 

Given a representation $\varphi: Y_{\nu}^{\geq }(\Glie) \longrightarrow \mathrm{End}(N)$ of the subalgebra $Y_{\nu}^{\geq}(\Glie)$, let us make sense of the formal series $(\mathrm{Id}\otimes \varphi)(\CR^-(z))$. For $\beta \in \BQ$, we define
$$ \mathscr{E}_{N,\beta} := \{ f \in \mathrm{End} (N) \ |\ \varphi(\xi_{i,-\langle \nu,\alpha_i\rangle} ) \circ f - f \circ \varphi(\xi_{i,-\langle \nu,\alpha_i\rangle} ) = (\beta,\alpha_i)f \quad \textrm{for $i \in I$} \}. $$
Clearly, the map $\varphi$ sends $Y_{\nu}^{\geq}(\Glie)_{\beta}$ to $\mathscr{E}_{N,\beta}$.
The sum of the subspaces $\mathscr{E}_{N,\beta}$, for $\beta \in \BQ$, is direct and forms a subalgebra of $\mathrm{End}(N)$ that contains the image of $\varphi$. Let $\mathscr{E}_N$ denote the resulting $\BQ$-graded algebra. Then $(\mathrm{Id}\otimes \varphi)(\CR^-(z))$ belongs to the $z$-completed tensor product $Y_{\mu}(\Glie) \otimes_z \mathscr{E}_N$.  Similar constructions work for $Y_{\nu}^{\leq}(\Glie)$-modules.

The next result follows essentially from \cite[\S 4.2]{GTLW}. We give a slightly different proof without any assumption on the representations.

\begin{prop}  \label{prop: stable uniqueness}
    Let $\mu$ and $\nu$ be two coweights. 
    \begin{itemize}
   \item[(1)] For $\varphi: Y_{\nu}^{\geq}(\Glie) \longrightarrow \mathrm{End}(N)$ a representation,  there exists a unique formal series
    \begin{gather*}
        \mathcal{A}(z) = \sum_{\beta \in \BQ_+} \mathcal{A}_{\beta}(z) \in Y_{\mu}(\Glie) \otimes_z \mathscr{E}_N \quad \textrm{with} \\
        \mathcal{A}_{\beta}(z) \in (Y_{\mu}(\Glie)_{-\beta} \otimes \mathscr{E}_{N,\beta})((z^{-1})) \quad \textrm{for}\quad \beta \in \BQ_+ \quad \textrm{and}\quad \mathcal{A}_0(z) = 1 \otimes \mathrm{Id}
    \end{gather*}
    which satisfies the intertwining equations in $(Y_{\mu}(\Glie)\otimes_z \mathscr{E}_N)((u^{-1}))$ for $i \in I$:
    $$ \mathcal{A}(z) \times (\tau_z \otimes \varphi) \circ \Delta_{\mu,\nu}(\xi_i(u)) = (\xi_i(u-z) \otimes \varphi(\xi_i(u))) \times \mathcal{A}(z). $$
    Such a formal series is precisely $(\mathrm{Id} \otimes \varphi)(\CR^-(z))$.
    \item[(2)] For $\varphi: Y_{\mu}^{\leq}(\Glie) \longrightarrow \mathrm{End} (M)$ a representation, there exists a unique formal series
    \begin{gather*}
        \mathcal{B}(z) = \sum_{\beta \in \BQ_+} S_{\beta}(z) \in \mathscr{E}_M \otimes_z Y_{\nu}(\Glie) \quad \textrm{with} \\
        \mathcal{B}_{\beta}(z) \in (\mathscr{E}_{M,-\beta} \otimes Y_{\mu}(\Glie)_{\beta})((z^{-1})) \quad \textrm{for}\quad \beta \in \BQ_+ \quad \textrm{and}\quad \mathcal{B}_0(z) = \mathrm{Id} \otimes 1
    \end{gather*}
    which satisfies the intertwining equations in $(\mathscr{E}_M \otimes_z Y_{\nu}(\Glie))((u^{-1}))$ for $i \in I$:
    $$ \mathcal{B}(z) \times (\varphi \otimes \tau_z)\circ \Delta_{\mu,\nu}(\xi_i(u)) = (\varphi(\xi_i(u)) \otimes \xi_i(u-z)) \times \mathcal{B}(z). $$
    Such a formal series is precisely $(\varphi \otimes \mathrm{Id})(\CR^-(-z))$.
    \end{itemize}
\end{prop}
\begin{proof}
We only prove Part (1); the same idea works for Part (2).
    Clearly, the power series $(\mathrm{Id} \otimes \varphi)(\CR^-(z))$ is unitriangular and satisfies the intertwining property as a consequence of the first intertwining property of Proposition \ref{prop: R- intertwining}. For uniqueness, let $A(z)$ be an element in $Y_{\mu}(\Glie) \otimes_z \mathscr{E}_N$ with the same weight decomposition and intertwining property as $\mathcal{A}(z)$, but with the initial condition $A_0(z) = 0$. It suffices to show that $A_{\beta}(z) = 0$ for $\beta \in \BQ_+$. We proceed by induction on the height of $\beta$ defined as $\mathrm{ht}(\beta) :=\langle \sum_{i\in I} \varpi_i^{\vee}, \beta \rangle \in \BN$. 

   For $i \in I$, let $s_i = -\langle \mu, \alpha_i\rangle$ and $t_i := -\langle \nu, \alpha_i\rangle$. Define also
\begin{equation*}  
    g_i^{\mu,\nu} := 1 \otimes \xi_{i,t_i+1} + \xi_{i,s_i+1} \otimes 1 + \xi_{i,s_i} \otimes \xi_{i,t_i} \in Y_{\mu}^0(\Glie)\otimes Y_{\nu}^0(\Glie). 
\end{equation*}
This is precisely the coefficient of $u^{-s_i-t_i-2}$ in $\xi_i(u) \otimes \xi_i(u)$. We have
$$ \Delta_{\mu,\nu}(\xi_{i,s_i+t_i+1}) = g_i^{\mu,\nu} - \sum_{\gamma \in \Phi} (\alpha_i,\gamma) x_{\gamma}^-\otimes x_{\gamma}^+, $$
and $(\tau_z \otimes \varphi)(g_i^{\mu,\nu})$ is the following polynomial in $z$ with coefficients in $Y_{\mu}(\Glie) \otimes \mathscr{E}_N$: 
   $$ 1 \otimes \varphi(\xi_{i,t_i+1}) + \left(\xi_{i,s_i+1} + (s_i+1)z \xi_{i,s_i}+\frac{s_i(s_i+1)}{2}z^2\right)\otimes \mathrm{Id} + (\xi_{i,s_i} + s_iz) \otimes \varphi(\xi_{i,t_i}).  $$
   In the $i$th intertwining equation for $A(z)$, we take the coefficient of $u^{-s_i-t_i-2}$ at both sides and then project to their $(-\beta,\beta)$-components. In view of the commutativity of $A(z)$ and $\xi_{i,s_i} \otimes \mathrm{Id} + 1 \otimes \varphi(\xi_{i,t_i})$, this results in:
   \begin{equation*}
   \begin{split}
        [A_{\beta}(z), 1 \otimes \varphi(\xi_{i,t_i+1}) + \xi_{i,s_i+1} \otimes \mathrm{Id} + \xi_{i,s_i} \otimes \varphi(\xi_{i,t_i})] + (\alpha_i,\beta)z A_{\beta}(z) \\
        = \sum_{\gamma \in \Phi} (\alpha_i,\gamma) A_{\beta-\gamma}(z) (x_{\gamma}^- \otimes \varphi(x_{\gamma}^+)).
   \end{split}          
   \end{equation*}
   Here we set $A_{\beta-\gamma}(z) = 0$ if $\beta - \gamma \notin \BQ_+$. Since $\mathrm{ht}(\beta-\gamma) < \mathrm{ht}(\beta)$, we apply the induction hypothesis to get that the right-hand side is zero. 
   For $n \in \BZ$, let $A_{\beta,n}$ denote the coefficient of $z^n$ in the Laurent series $A_{\beta}(z)$. Then $A_{\beta,n} = 0$ for $n \gg 0$, and the above equation becomes
   $$ [A_{\beta,n}, 1 \otimes \varphi(\xi_{i,t_i+1}) + \xi_{i,s_i+1} \otimes \mathrm{Id} + \xi_{i,s_i} \otimes \varphi(\xi_{i,t_i})] + (\alpha_i,\beta) A_{\beta,n-1} = 0 \quad \textrm{for $i \in I$ and $n \in \BZ$}. $$
   Assume that $A_{\beta,n} = 0$. Then $(\alpha_i,\beta) A_{\beta,n-1} = 0$ for all $i \in I$. Since $\beta \neq 0$, there exists $i \in I$ such that $(\alpha_i,\beta) \neq 0$. It follows that $A_{\beta,n-1} = 0$. Another decreasing induction on $n$ shows that $A_{\beta,n} = 0$ for all $n \in \BZ$. We conclude that $A_{\beta}(z) = 0$. 
\end{proof}

\section{Conjugation formulas for shifted Yangians} \label{se: conjugation formula shifted Yangians}
In this section, we establish conjugation formulas for the evaluations of the unitriangular R-matrix at a large family of representations of shifted Yangians.

Recall from the co-ideal subalgebra property of Lemma \ref{lem: Yangian coproduct estimation} that the tensor product of a $Y_{\mu}(\Glie)$-module with a $Y_{\nu}^{\geq}(\Glie)$-module is naturally a $Y_{\mu+\nu}^{\geq}(\Glie)$-module.
\begin{lem} \label{lem: trivilization}
    Let $\nu$ be a coweight and $N$ be a finite-dimensional $Y_{\nu}^{\geq}(\Glie)$-module. Then, there exists a one-dimensional module $K$ over another shifted Yangian $Y_{\epsilon}(\Glie)$ such that the tensor product $K\otimes N$ is a trivial $Y_{\epsilon+\nu}^{\geq}(\Glie)$-module. 
\end{lem}
\begin{proof}
 Let $\varphi: Y_{\nu}^{\geq}(\Glie) \longrightarrow \mathrm{End}(N)$ be the structural map of the module $N$. For each $i \in I$, the vectors $\varphi(x_{i,n}^+)$ for $n \in \BN$ are linearly dependent in the finite-dimensional vector space $\mathrm{End}(N)$. There exist coefficients $c_0, c_1, \cdots, c_s \in \BC$ such that $c_s = 1$ and  
    $$ c_0 \varphi(x_{i,0}^+) + c_1 \varphi(x_{i,1}^+) + \cdots + c_s \varphi(x_{i,s}^+) = 0. $$
    Applying repeatedly the following commutation relation in $Y_{\nu}^{\geq}(\Glie)$:
    $$[\xi_{i,-\langle \nu,\alpha_i\rangle+1} - \frac{1}{2} \xi_{i,-\langle \nu,\alpha_i\rangle}^2, x_{i,n}^{\pm}] = \pm 2d_i x_{i,n+1}^{\pm}$$
     we get the following relation for $n \in \BN$,
    $$ c_0 \varphi(x_{i,n}^+) + c_1 \varphi(x_{i,n+1}^+) + \cdots + c_s \varphi(x_{i,n+s}^+) = 0. $$
    This is precisely the coefficient of $u^{-n-1}$ in the product of $\varphi(x_i^+(u))$ with the polynomial $\mathbf{p}_i(u) := c_0+c_1u+\cdots +c_s u^s$, which is monic by our assumption $c_s = 1$.  

We have obtained a tuple $(\mathbf{p}_i(u))_{i\in I}$ of monic polynomials in $u$ such that the power series $\langle \mathbf{p}_i(u) \varphi(x_i^+(u))\rangle_+$ is identically zero for each $i \in I$.
    Let $K$ be the corresponding one-dimensional module over a shifted Yangian $Y_{\epsilon}(\Glie)$.  By Eq.\eqref{equ: trivial times}, in the $Y_{\epsilon+\nu}^{\geq}(\Glie)$-module $K\otimes N$ whose underlying space is identified with $N$, each generating series $x_i^+(u)$, for $i \in I$, acts as $\langle \mathbf{p}_i(u) \varphi(x_i^+(u))\rangle_+ = 0$. Therefore, $K\otimes N$ is a trivial module.
\end{proof}
In the situation of the lemma, we call the one-dimensional module $K$ a {\it left trivializer} of $N$. Similarly, if $M$ is a finite-dimensional $Y_{\mu}^{\leq}(\Glie)$-module, then there exists a one-dimensional module $K$ over a shifted Yangian $Y_{\epsilon}(\Glie)$ such that the tensor product $M\otimes K$ is a trivial $Y_{\mu+\epsilon}^{\leq}(\Glie)$-module. In this case, we call $K$ a {\it right trivializer} of $M$.

\begin{rem}
   Assume $N$ is a module over the full shifted Yangian $Y_{\nu}(\Glie)$. By considering also $x_i^-(u)$ in the proof, we obtain another one-dimensional module $L$ over a shifted Yangian $Y_{\eta}(\Glie)$ such that the triple tensor product $(K\otimes N) \otimes L$ is a trivial module over $Y_{\epsilon+\nu+\eta}(\Glie)$. This gives a representation-theoretic meaning to the generalized Baxter relations \cite[Corollary 4.2]{HZ} in the Grothendieck ring of finite-dimensional representations of shifted Yangians by noticing that the isomorphism class of a finite-dimensional trivial module is a sum of isomorphism classes of one-dimensional modules. 
\end{rem}

For $\mu$ a coweight, recall from Eqs.\eqref{equ: psi-}--\eqref{equ: psi+} the two $\BC[z]$-algebra homomorphisms $\mathrm{C}_{K,z}^{\pm}$ from $Y_{\mu}^{\pm}(\Glie)[z]$ to $Y_{\mu}^{\pm}(\Glie)((z^{-1}))$. They are defined by conjugations with R-matrices associated to a one-dimensional module $K$ of another shifted Yangian. 

Our second main result of the paper is a Yangian analog of the first and third conjugation formulas in Theorem \ref{thm: conjugation formula qaf}.

\begin{theorem}  \label{thm: main result}
    Let $\mu, \epsilon$ and $\nu$ be three coweights and $K$ be a one-dimensional module over the shifted Yangian $Y_{\epsilon}(\Glie)$.
    \begin{itemize}
        \item[(1)] Given a finite-dimensional representation $\varphi: Y_{\nu}^{\geq}(\Glie) \longrightarrow \mathrm{End}(N)$ which is left trivialized by $K$, we have 
    $$ (\mathrm{Id} \otimes \varphi)(\CR^-(z)) = (\mathrm{C}_{K,z}^- \circ \tau_z \otimes \varphi) (\Theta_K^{\mu,\nu}(0)). $$
    \item[(2)] Given a finite-dimensional representation $\varphi: Y_{\mu}^{\leq}(\Glie) \longrightarrow \mathrm{End}(M)$ which is right trivialized by $K$, we have 
    $$ (\varphi \otimes \mathrm{Id})(\CR^-(-z)) = (\varphi \otimes \mathrm{C}_{K,z}^+ \circ \tau_z) (\Theta_K^{\mu,\nu}(0)^{-1}). $$ 
    \end{itemize}
\end{theorem}
\begin{proof}
    We will mainly prove Part (1), and sketch Part (2) in the end. Let $\mathcal{A}(z)$ denote the right-hand side. We need to show that $\mathcal{A}(z)$ satisfies the unitriangularity and intertwining property of Proposition \ref{prop: stable uniqueness}(1), so that is equal to $(\mathrm{Id} \otimes \varphi)(\CR^-(z))$. 

    First, notice that  $\mathcal{A}(z)$ is the sum over $\beta \in \BQ_+$ of the formal series
$$\mathcal{A}_{\beta}(z) :=  (\mathrm{C}_{K,z}^- \circ \tau_z \otimes \varphi)(\Theta_{K,\beta}^{\mu,\nu}(0)). $$
By Theorem \ref{thm: Theta}, each $\Theta_{K,\beta}^{\mu,\nu}(0)$ lies in the ordinary tensor product $Y_{\mu}^-(\Glie)_{-\beta} \otimes Y_{\nu}^+(\Glie)_{\beta}$. Applying the algebra homomorphism $\mathrm{C}_{K,z}^- \circ \tau_z: Y_{\mu}^-(\Glie) \longrightarrow Y_{\mu}^-(\Glie)((z^{-1}))$ to the first tensor factor, we see that the formal series $\mathcal{A}_{\beta}(z)$ belongs to the tensor product $Y_{\mu}^-(\Glie)_{-\beta}((z^{-1})) \otimes \mathscr{E}_{N,\beta}$ and is therefore a Laurent series in $z^{-1}$ with coefficients in $Y_{\mu}^-(\Glie)_{-\beta} \otimes \mathscr{E}_{N,\beta}$. Together with $\mathcal{A}_0(z) = 1\otimes 1$, we obtain the unitriangularity of $\mathcal{A}(z)$.

 The proof of intertwining property for $\mathcal{A}(z)$ is in the spirit of \cite[Theorem 5.6]{Z2}.
    To simplify notations, let $V$ denote the regular representation of the shifted Yangian $Y_{\mu}(\Glie)$ on itself. Let $V[z]$ be the $Y_{\mu}(\Glie)[z]$-module by scalar extension, and then take its pullback along the algebra homomorphism $\tau_z: Y_{\mu}(\Glie) \longrightarrow Y_{\mu}(\Glie)[z]$. The resulting $Y_{\mu}(\Glie)$-module, denoted by $V_z$, is called a deformed module. 

For the one-dimensional module $K$ over the shifted Yangian $Y_{\epsilon}(\Glie)$, let $(\mathbf{p}_i(u))_{i\in I}$ denote its $\ell$-weight and $\rho: Y_{\epsilon}(\Glie) \longrightarrow \BC$ denote its structural map.

    \medskip 
    
    \noindent {\bf Step 1: completed triple tensor product modules.}
     Consider the triple tensor product module $K \otimes (V_z \otimes N)$ over $Y_{\epsilon+\mu+\nu}^{\geq}(\Glie)$. It can be obtained similarly as the above deformed module in two steps. First, let $(V\otimes N)[z]$ be the $Y_{\mu}(\Glie) \otimes Y_{\nu}^{\geq}(\Glie)[z]$-module by scalar extension. Then, take its pullback along the following algebra homomorphism
   $$ (\rho \otimes \tau_z \otimes \mathrm{Id}) \circ \Delta_{\epsilon(\mu\nu)}: Y_{\epsilon+\mu+\nu}^{\geq}(\Glie) \longrightarrow Y_{\epsilon}(\Glie) \otimes Y_{\mu}(\Glie) \otimes Y_{\nu}^{\geq}(\Glie) \longrightarrow Y_{\mu}(\Glie) \otimes Y_{\nu}^{\geq}(\Glie)[z]. $$
   These two steps can be carried out by replacing the underlying space $V\otimes N[z]$ with the following larger space of Laurent series in $z^{-1}$, based on the weight grading of $Y_{\mu}(\Glie)$:
   $$ \mathscr{X} := \sum_{\beta \in \BQ} (Y_{\mu}(\Glie)_{\beta} \otimes N)((z^{-1})) \subset (Y_{\mu}(\Glie) \otimes N)((z^{-1})). $$
   The resulting $Y_{\epsilon+\mu+\nu}^{\geq}(\Glie)$-module structure on $\mathscr{X}$ is denoted by $\overline{K\otimes (V_z \otimes N)}$ to indicate that it contains the ordinary triple tensor product $K \otimes (V_z \otimes N)$ as a submodule. Similarly, on the same space $\mathscr{X}$, we define the following $Y_{\epsilon+\mu+\nu}^{\geq}(\Glie)$-module structures as completions of the corresponding triple tensor product modules, which make sense because the first two tensor factors are modules over full shifted Yangians:
   $$ \overline{(K \otimes V_z) \otimes N}, \quad \overline{(V_z \otimes K) \otimes N}, \quad \overline{V_z \otimes (K\otimes N)}. $$

\medskip

\noindent {\bf Step 2: module isomorphisms.}  Since the one-dimensional module $K$ is trivial, by Theorem \ref{thm: asso trivial} and Remark \ref{rem: coideal associativity}(1), the two modules $\overline{K\otimes (V_z\otimes N)}$ and $\overline{(K\otimes V_z) \otimes N}$ are the same. Since $N$ is finite-dimensional, $\mathscr{X}$ is identified with the ordinary tensor product space $V^z \otimes N$. By Theorem \ref{thm: one-dim R}, we have a module isomorphism
   $$ \CR_{K,V}(z) \otimes \mathrm{Id}_N: \overline{(K\otimes V_z) \otimes N} \longrightarrow \overline{(V_z \otimes K) \otimes N}. $$
   By Theorem \ref{thm: Theta} and Remark \ref{rem: coideal associativity}(2), the Theta series defines a module isomorphism
   $$ \Theta_K^{\mu,\nu}(0)|_{V_z, N}: \overline{(V_z \otimes K) \otimes N} \longrightarrow \overline{V_z \otimes (K\otimes N)}. $$
  We obtain the following module isomorphism, denoted by $\Phi(z) \in \mathrm{Aut}(\mathscr{X})$:
   $$ \Theta_K^{\mu,\nu}(0)|_{V_z, N} \circ (\CR_{K,V}(z) \otimes \mathrm{Id}_N): \overline{K\otimes (V_z \otimes N)} \longrightarrow \overline{V_z \otimes (K\otimes N)}. $$
      Recall from Theorem \ref{thm: Theta} that the first tensor factor of $\Theta_K^{\mu,\nu}(0)$ lies in $Y_{\mu}^-(\Glie)$, on which the conjugation by $\CR_{K,V}(z)^{-1}$ defines the map $\mathrm{C}_{K,z}^-$ of \eqref{equ: psi-}. This implies
   $$ \Phi(z)  = (\CR_{K,V}(z) \otimes \mathrm{Id}) \circ (\mathrm{Ad}_{\CR_{K,V}(z)}^{-1} \otimes \mathrm{Id})(\Theta_K^{\mu,\nu}(0)|_{V_z,N}) = (\CR_{K,V}(z) \otimes \mathrm{Id}) \circ \mathcal{A}(z). $$
   Since $\CR_{K,V}(z)$ commutes with the action of the Drinfeld--Cartan subalgebra $Y_{\mu}^0(\Glie)$ on $V^z$ by left multiplication, it suffices to prove the same intertwining property of Proposition \ref{prop: stable uniqueness}(1) with $\mathcal{A}(z)$ replaced by $\Phi(z)$ everywhere.
   
\medskip

\noindent{\bf Step 3: actions of Drinfeld--Cartan series.}
   Consider the module isomorphism $$\Phi(z) : \overline{K\otimes (V_z \otimes N)}\longrightarrow \overline{V_z \otimes (K\otimes N)}.$$ 
   We compute the actions of the Drinfeld--Cartan series $\xi_i(u)$ of $Y_{\epsilon+\mu+\nu}^{\geq}(\Glie)$ on both modules as Laurent series in $u^{-1}$ with coefficients in $\mathrm{End}(\mathscr{X})$. Since $K$ is a trivial module, by Remark \ref{rem: coideal}, the action of $\xi_i(u)$ on the first module is $\mathbf{p}_i(u) (\tau_z\otimes \varphi)(\Delta_{\mu,\nu}(\xi_i(u)))$. 

   Since $K\otimes N$ is a trivial $Y_{\epsilon+\nu}^{\geq}(\Glie)$-module, by Remark \ref{rem: coideal}, the action of $\xi_i(u) \in Y_{\mu+\nu+\nu}^{\geq}(\Glie)[[u^{-1}]]$ on $V_z \otimes (K\otimes N)$ is precisely the tensor product of $\xi_i(u) \in Y_{\mu}(\Glie)[[u^{-1}]]$ acting on $V_z$ and $\xi_i(u) \in Y_{\epsilon+\nu}^{\geq}(\Glie)[[u^{-1}]]$ acting on $K\otimes N$. The former is multiplication by $\xi_i(u-z)$. Again by Remark \ref{rem: coideal}, the latter is  $\Bp_i(u) \varphi(\xi_i(u))$ after identifying the underlying spaces $N$ and $K\otimes N$. In summary, the action of $\xi_i(u)$ on $V_z \otimes (K\otimes N)$ and therefore on the second module is $\xi_i(u-z) \otimes \mathbf{p}_i(u)\varphi(\xi_i(u))$.
  
   As a consequence of $\Phi(z)$ being a module morphism, we have 
   $$ \Phi(z) \circ \mathbf{p}_i(u) (\tau_z \otimes \varphi)(\Delta_{\mu,\nu}(\xi_i(u))) = (\xi_i(u-z) \otimes \mathbf{p}_i(u) \varphi(\xi_i(u))) \circ \Phi(z) \in \mathrm{End}(\mathscr{X})((u^{-1})).  $$
   Dividing both sides by the monic polynomial $\mathbf{p}_i(u)$, we get the desired $i$th intertwining equation for $\Phi(z)$ in Proposition \ref{prop: stable uniqueness}(1).

   \medskip

   For Part (2), we modify Steps 1--2 by considering the following completed triple tensor product modules over $Y_{\mu+\nu+\epsilon}^{\leq}(\Glie)$, all defined on the same underlying space:
   $$ \overline{(M \otimes Y_{\nu}(\Glie)_z)\otimes K}, \quad \overline{M \otimes (Y_{\nu}(\Glie)_z \otimes K)}, \quad \overline{M \otimes (K \otimes Y_{\nu}(\Glie)_z)}, \quad \overline{(M\otimes K) \otimes Y_{\nu}(\Glie)_z}. $$
   They are all isomorphic: identify map from the first to second, $\mathrm{Id} \otimes \CR_{K,Y_{\nu}(\Glie)}(z)^{-1}$ from the second to third, and $\Theta_K^{\mu,\nu}(0)^{-1}$ evaluated at $M \otimes Y_{\nu}(\Glie)_z$ from the third to fourth. Their composition defines an isomorphism from the first module to the fourth module which is factorized as follows by Eqs.\eqref{equ: conjugation T Yangian} and \eqref{equ: psi+}:  
   $$ (\mathrm{Id} \otimes \mathcal{R}_{K,Y_{\nu}(\Glie)}(z)^{-1}) \circ (\varphi\otimes \mathrm{C}_{K,z}^+\circ \tau_z)(\Theta_K^{\mu,\nu}(0)^{-1}). $$
  We are reduced to the intertwining equations of Proposition \ref{prop: stable uniqueness}(2) for the composition. This can be done by almost the same arguments as in Step 3.
\end{proof}
Based on Remark \ref{rem: Theta Omega} and Proposition \ref{prop: R- intertwining}, we can prove Theorem \ref{thm: coproduct shifted Yangian} purely algebraically as Theorem \ref{thm: conjugation formula qaf} by introducing suitable quotients of $Y^{\pm}(\Glie)$. The above proof is representation-theoretical, and it does not work for shifted quantum affine algebras \cite{FT} because their Drinfeld--Jimbo coproduct is unknown beyond type A.
\begin{example}
This is a Yangian analog of Example \ref{ex: Ding-Frenkel}. Let $\Glie = \mathfrak{sl}_{r+1}$. On the vector space $\BC^{r+1}$ there is a representation $\varphi$ of the Yangian $Y(\Glie)$ given by
$$ \varphi(x_{i,n}^+) = (\frac{i}{2})^n E_{i,i+1}, \quad \varphi(x_{i,n}^-) = (\frac{i}{2})^n E_{i+1,i} \quad \textrm{for $1\leq i \leq r$.} $$
It is left and right trivialized  by the one-dimensional module $K$ of $\ell$-weight $(u-\frac{i}{2})_{1\leq i \leq r}$. We argue as in Example \ref{ex: Theta qaf} to get $\Theta_{K,\alpha_i}(z) = \Theta_{i,\alpha_i}(z+\frac{i}{2})$ and
\begin{align*}
    (\mathrm{Id}\otimes \varphi)(\CR^-(z)) &= (\mathrm{C}_{K,z}^- \circ \tau_z \otimes \varphi)(x_{i,0}^- \otimes x_{i,0}^+) = - x_i^-(\frac{i}{2}-z) \otimes  E_{i,i+1}, \\
    (\varphi \otimes \mathrm{Id})(\CR^-(-z)) &= (\varphi \otimes \mathrm{C}_{K,z}^+ \circ \tau_z)(-x_{i,0}^- \otimes x_{i,0}^+) = E_{i+1,i} \otimes x_i^+(\frac{i}{2}-z).
\end{align*}
The generating series $x_i^{\pm}(z-\frac{i}{2})$ are recovered as off-diagonal entries of the unitriangular L-operators associated with $\rho$. This agrees with the Yangian analog of the Ding--Frenkel homomorphism \cite[(5.26)--(5.27)]{BK}. In view of the R-matrix realization of the Yangian \cite[Theorem 6.2]{Wendlandt}, the two unitriangular matrices of the Gauss decomposition in \cite[(5.1)]{BK} are $(\mathrm{Id}\otimes \varphi)(\CR^-(-z))$ and $(\varphi\otimes \mathrm{Id})(\CR^-(z))^{-1}$.
\end{example}

We conclude this section with rationality of evaluations of $\CR^-(z)$ at finite-dimensional representations. It is a shifted version of \cite[Theorem 4.1(3), first point]{GTLW}.

\begin{prop}  \label{prop: rationality R-}
    Let $\varphi: Y_{\mu}^{\leq}(\Glie) \longrightarrow \mathrm{End}(M)$ and $\psi: Y_{\nu}^{\geq}(\Glie) \longrightarrow \mathrm{End}(N)$ be two finite-dimensional representations. Then the power series $(\varphi \otimes \psi)(\CR^-(z))$ is the Taylor expansion around $z=\infty$ of an $\mathrm{End}(M\otimes N)$-valued rational function. Moreover, each pole of this rational function is necessarily of the form $b-a$ where $a$ is a pole of $\varphi(x_i^-(u))$ and $b$ is a pole of $\psi(x_i^+(u))$ for certain $i \in I$.
\end{prop}
\begin{proof}
    The idea is almost the same as  Proposition \ref{prop: rationality qaf}. For $i \in I$, by Lemma \ref{lem: trivilization}, the power series $\psi(x_i^+(u))$ is rational and admits a unique denominator $\Bp_i(u) \in \BC[u]$ as a monic polynomial. Let $K$ be the one-dimensional module of $\ell$-weight $(\Bp_i(u))_{i\in I}$, so that it is a left trivializer of $\psi$. Then, by Theorem \ref{thm: main result}(1) and polynomiality of Theta series and, it is enough to prove the rationality of $\varphi\circ \mathrm{C}_{K,z}^- (x_{i,m}^-)$ for all $i \in I$ and $m \in \BN$. If $\Bp_i(u) = 1$, this is trivial. Otherwise, take the partial fraction decomposition
    $$ \frac{1}{\Bp_i(u)} = \sum_{(b,n) \in \Lambda} \frac{\lambda_{b,n}}{(u-b)^{n+1}} $$
    where $\Lambda \subset \BC\times \BN$ is finite and $\lambda_{b,n} \in \BC^{\times}$ for $(b,n) \in \Lambda$. We have
    \begin{align*}
        \mathrm{C}_{K,z}^-(x_{i,m}^-) &= \frac{1}{\Bp_i(z+\sigma_i^-)}(x_{i,m}^-) = \sum_{(b,n)\in \Lambda} \frac{\lambda_{b,n}}{(z+\sigma_i^--b)^{n+1}}(x_{i,m}^-) \\
        &= \sum_{(b,n)\in \Lambda} \frac{(-1)^n \lambda_{b,n}}{n!} \partial_z^n \left(\frac{1}{z-b+\sigma_i^-} (x_{i,m}^-) \right) \\
        &= \sum_{(b,n)\in \Lambda} \frac{(-1)^{n+1}\lambda_{b,n}}{n!} \partial_z^n\left((b-z)^m x_i^-(b-z) - \sum_{k=0}^{m-1} x_{i,k}^-(b-z)^{m-k-1} \right).
    \end{align*}
   Then conclude from the rationality of $\varphi(x_i^-(z))$ in Lemma \ref{lem: trivilization}.
\end{proof}

Each $Y_{\nu}(\Glie)$-module in the category $\mathcal{O}_{\nu}$ of \cite[\S 3.3]{HZ} is a union of finite-dimensional sub-$Y_{\nu}^{\geq}(\Glie)$-modules, so Theorem \ref{thm: main result} and Proposition \ref{prop: rationality R-} are applicable.

\end{document}